\documentclass[reqno,11pt]{amsart}

\numberwithin{equation}{section}

\usepackage[utf8]{inputenc}
\usepackage[english]{babel}

\usepackage{mathrsfs}
\usepackage{mathtools}
\usepackage{amsmath}
\usepackage{amssymb}

\usepackage{amsmath}
\usepackage{amsfonts,amssymb}
\usepackage{booktabs}
\usepackage{esint} 
\usepackage{amsthm}

\usepackage{nicefrac}

\usepackage{comment}
\usepackage{todonotes}
\setuptodonotes{inline} 

\usepackage{accents}
\usepackage{epsfig}
\usepackage{graphicx}
\usepackage{csquotes}

\usepackage{psfrag}
\usepackage{graphicx}
\usepackage{graphpap,latexsym,epsf}
\usepackage{color}
\usepackage{amssymb,mathrsfs,enumerate}
\usepackage[colorlinks, citecolor=citegreen, linkcolor=refred]
{hyperref}

\usepackage{caption}
\usepackage{subcaption}
\usepackage{float}

\usepackage{enumitem}

\usepackage[square,sort,comma,numbers]{natbib}

\makeatletter
\def\@tocline#1#2#3#4#5#6#7{\relax
\ifnum #1>\c@tocdepth 
\else
\par \addpenalty\@secpenalty\addvspace{#2}%
\begingroup \hyphenpenalty\@M
\@ifempty{#4}{%
\@tempdima\csname r@tocindent\number#1\endcsname\relax
}{%
\@tempdima#4\relax
}%
\parindent\z@ \leftskip#3\relax \advance\leftskip\@tempdima\relax
\rightskip\@pnumwidth plus4em \parfillskip-\@pnumwidth
#5\leavevmode\hskip-\@tempdima
\ifcase #1
\or\or \hskip 1em \or \hskip 2em \else \hskip 3em \fi%
#6\nobreak\relax
\dotfill\hbox to\@pnumwidth{\@tocpagenum{#7}}\par
\nobreak
\endgroup
\fi}
\makeatother
\makeatletter
\newcommand{\xdashrightarrow}[2][]{\ext@arrow 0359\rightarrowfill@@{#1}{#2}}
\def\rightarrowfill@@{\arrowfill@@\relax\relbar\rightarrow}
\def\arrowfill@@#1#2#3#4{%
$\m@th\thickmuskip0mu\medmuskip\thickmuskip\thinmuskip\thickmuskip
\relax#4#1
\xleaders\hbox{$#4#2$}\hfill
#3$%
}
\makeatother
\usepackage{hyperref}

\newcommand{\R}{\mathbb{R}}

\newcommand{\N}{\mathbb{N}}

\newcommand{\bfc}{\mathbf{c}}

\newcommand{\Om}{\Omega}

\newcommand{\ove}{\overline}

\newcommand{\rmd}{{\rm d}}

\mathchardef\emptyset="001F

\newcommand*{\bigchi}{\mbox{\Large$\chi$}}

\definecolor{vgreen}{rgb}{0.1,0.5,0.2}
\definecolor{viola}{RGB}{85,26,139}
\definecolor{citegreen}{rgb}{0,0.6,0}
\definecolor{refred}{rgb}{0.8,0,0}

\newtheorem{thm}{Theorem}[section]

\newtheorem{lemma}[thm]{Lemma}

\newtheorem{remarkr}[thm]{Remark}
\newenvironment{remark}
  {\pushQED{\qed}\remarkr}
  {\popQED\endremarkr}

\theoremstyle{definition}
\newtheorem{example}[thm]{Example}

\makeindex{}

\makeatletter
\renewcommand*\env@matrix[1][*\c@MaxMatrixCols c]{%
  \hskip -\arraycolsep
  \let\@ifnextchar\new@ifnextchar
  \array{#1}}
\makeatother

\begin{document}

\title{Mass optimization problems with convex cost}

\author[G.~Buttazzo]{Giuseppe Buttazzo}
\author[M.S.~Gelli]{Maria Stella Gelli}
\author[D.~Lu\v ci\' c]{Danka Lu\v ci\' c}

\begin{abstract} 
In this paper we consider a mass optimization problem in the case of scalar state functions, where instead of imposing a constraint on the total mass of the competitors, we penalize 
the classical compliance by a convex functional defined on the space of measures. We obtain a characterization of optimal solutions to the problem through a suitable PDE. This generalizes the case considered in the literature of a linear cost and applies to the optimization of a conductor where very low and very high conductivities have both a high cost, and then the study of nonlinear models becomes relevant.
\end{abstract}

\makeatletter

\date{\today} 

\keywords{Mass optimization problem, Convex functionals on measures, Sobolev spaces with respect to measures, Fenchel duality} 
\subjclass[2020]{49J45, 49K20, 49J20, 28A50, 28A33, 46A20, 46E36}

\maketitle


\section[Introduction]{Introduction}\label{sintro}
An optimization problem that plays a central role in many questions in Applied Mathematics is the so-called 
\emph{mass optimization problem}. A version of such a problem, in the \emph{scalar case}, provides a 
mathematical framework for the study of stationary heat conduction models, for instance in finding optimal 
mixtures of two conductors (see for example \cite{MT}).
It has been studied in the celebrated paper \cite{BB2001} and reads as follows: let \(\Omega\subset\R^n\) be 
a bounded Lipschitz domain, whose closure represents a given design region, and let a signed measure \(f\in 
\mathscr M(\overline \Omega)\) with finite total variation represent a given heat source density.

The energy associated to some distribution \(\mu\in\mathscr M^+(\overline\Omega)\) of conducting material is 
given by:
\begin{equation}\label{eq:compliance}
\mathcal E_f(\mu)\coloneqq \inf 
\left\{\int \frac{|\nabla u|^2}{2}\,\rmd \mu-\langle f,u\rangle:\, u\in \mathscr D(\Omega)\right\},
\end{equation}
where $\mathscr D(\Omega)$ is the class of smooth functions compactly supported in $\Omega$. The optimization
problem one wants to consider is that of finding a distribution \(\mu\), of a given total amount \(m>0\) of
material, which provides the minimal \emph{compliance} \(\mathcal J_f(\mu)\), defined as  \(\mathcal 
J_f(\mu)\coloneqq-\mathcal E_f(\mu)\). Namely:
\begin{equation}\tag{\({\rm MOP}_f\)}\label{eq:MOP}
\min\big\{ \mathcal J_f(\mu):\,\mu\in \mathscr M^+(\overline \Omega),\ \mu(\overline \Omega)= m\big\}.
\end{equation}

The term ``scalar case'' associated to this problem is related to the fact that the competitors 
\(u\in\mathscr D(\Omega)\) in the minimization problem \eqref{eq:compliance} take values in \(\R\), and 
represent the temperature profiles. 
In \cite{BB2001}, the characterization of the optimal masses for \eqref{eq:MOP} 
is shown to be related to the Monge-Kantorovich PDE:
\begin{align}
-\textrm{div}(\mu\nabla_\mu u)= f,\quad|\nabla_\mu u|=1\ \mu\text{-a.e.,}\quad 
u \in \textrm{LIP}_{0,1}(\Omega).\tag{PDE}\label{eq:PDE}
\end{align}

In the above equation the notation \(\nabla_\mu u\) stands for the \(\mu\)-tangential gradient of \(u\), 
introduced in \cite{BBS}; we recall its precise definition in Subsection \ref{sec:Sobolev}, while the notion 
of \(\mu\)-divergence \({\rm div}(\mu\nabla_\mu u)\) is recalled in Section \ref{sec:Existence} (see \eqref{eq:mu_divergence}); for the definition of 
the space \({\rm LIP}_{0,1}(\Omega)\) see \eqref{eq:LIP}.  
Let us mention at this point that, as it is evident from the above PDE, the study and the characterization 
of the solutions to the problem \eqref{eq:MOP} rely on the notion of Sobolev functions with respect to an 
arbitrary measure \(\mu\) on \(\R^n\). Such a notion has been introduced for the first time in \cite{BBS}, 
and has been followed up to now by the new results concerning Sobolev and BV theory in this framework 
\cite{BF_Second_order,LPR20,BBF,GL,Zhikov}, as well as the applications in mass optimization problems (see 
\cite{VarMethods}), in homogenization theory (see \cite{Mandallena}), or in optimal transport (see 
\cite{Louet}), just to name a few.

For the forthcoming discussion, let us set
\begin{equation}\tag{\({\rm AP}_{f}\)}\label{eq:aux_pb_1}
\mathrm{I}_f\coloneqq\inf\big\{-\langle f,u\rangle:\, u\in \mathrm{LIP}_{0,1}(\Omega)\big\}.
\end{equation}
\noindent  It is proven in  \cite{BB2001} that  \eqref{eq:MOP} 
admits a solution whenever \(f\in \mathscr M(\overline\Omega)\)
and the following fact holds:
\begin{equation}\label{eq:OC}
\mu \text{ solves } \eqref{eq:MOP} \text{ and } u\text{ solves } \eqref{eq:aux_pb_1}
\quad
\text{ iff }\quad (\mu, u) \text{ satisfy } \eqref{eq:PDE}.
\end{equation}
The theory has been then extended to more general settings, some of which are listed below. 
\begin{itemize}
\item[-] Problem \eqref{eq:MOP} has been considered also in the framework of Riemannian manifolds (see 
\cite{Pratelli2005}).
 
\item[-] The optimizer \(\mu\) has been characterized via \eqref{eq:PDE} also in the \emph{vectorial case} 
(the term ``vectorial'' indicating that the competitors \(u\) in \eqref{eq:compliance} take values in 
\(\R^n\)), where also the Dirichlet regions (namely, the closed subsets of \(\overline\Omega\) on which the 
Dirichlet boundary conditions may be imposed) have been taken into consideration (see \cite{BB2001}). 

\item[-] Mass optimization problems in the vectorial case have been used in structural mechanics in order to 
find a distribution of a given amount of an elastic material which, for a given system of loads, gives the 
best resistance in terms of minimal compliance (see for instance \cite{AK}).
 
\item[-] In \cite{BF2007} a variant of mass optimization problem, involving an arbitrary linear operator \(A\) 
defined on the space of smooth functions \(\mathscr D(\Omega)\) instead of the gradient operator, has been 
considered, together with the applications to the elasticity theory of thin plates.
 
\item[-] In \cite{BBol2021}, instead of looking for the optimal mass distributions \(\mu\in \mathscr 
M(\overline\Omega)\), the problem of looking for an optimal conductivity tensor \(\sigma\in \mathscr 
M(\overline\Omega,\R^{n\times n}_{\mathrm{sym}})\) has been addressed.

\item[-] Integrands more general than \(|\cdot|^2\) have been considered in the definition of 
\(\mathcal E_f\), satisfying a suitable \(p\)-growth conditions with \(p>1\) (see \cite{BB2001}).

\item[-] Without any intent of being complete, we mention also some other recent contributions to the topic, 
in which different variants, motivated by some precise applications, have been considered (see \cite{BL}, 
\cite{LRZ}).
\end{itemize}

In this paper we provide a generalization of problem \eqref{eq:MOP} in the scalar case, in a \emph{new} and 
\emph{different} direction with respect to the above mentioned papers in the literature. Namely, instead of 
imposing a constraint on the total mass for the competitors in \eqref{eq:MOP}, we look at the compliance functional 
\(-\mathcal E_f\) penalized by a convex functional \(\mathcal C\) defined on the space of measures. This 
corresponds to rephrase the constraint in terms of Lagrangian multipliers (for a related discussion see 
\cite{Zalinescu}). More precisely, let \(\mathcal C\colon \mathscr M(\overline\Omega)\to[0,+\infty]\) be a 
\emph{cost functional} defined as
\[
\mathcal C(\mu)\coloneqq \int_{\Omega}c(a(x))\,\rmd x+c^{\infty}(1)\,\mu^s(\overline\Omega), 
\quad \text{ for every }\mu\in \mathscr M^+(\overline\Omega),
\]
where $\mu=a\,\rmd x+\mu^s$ is the decomposition of $\mu$ into absolute continuous part and singular part with respect to the Lebesgue measure and $c^{\infty}$ is the \emph{recession function} associated to  $c$ (cf.\ Subsection 
\ref{sec:FunctionalsOnMeasures}). The function \(c\colon \R\to [0,+\infty]\) above is referred to as a (homogeneous) \emph{cost function}; it is assumed to be proper, convex, lower semicontinuous, and to satisfy the following properties:
\begin{equation}\label{propc}
c(t)\ge\alpha t+\beta\,\text{ for all }t>0,\qquad c(t)=+\infty\text{ for all }t<0,
\end{equation}
with $\alpha>0$ and $\beta\in\R$. In order to rule out the trivial case when $c$ is finite only at the origin, we also require
\begin{equation}\label{tzero}
\text{there exists $t_0>0$ such that }c(t_0)<+\infty.
\end{equation}
We thus consider the following minimization problem, to which we refer to as to a \emph{mass optimization 
problem with convex cost}:
\begin{equation}\tag{\({\rm MOP}_{f,c}\)}\label{eq:MOPc}
\min\big\{-\mathcal E_f(\mu)+\mathcal C(\mu):\,\mu\in \mathscr M^+(\overline \Omega)\big\}.
\end{equation}
Roughly speaking, by adding the penalization term given by the functional \(\mathcal C\) we are taking into 
account the cost to select conducting properties of the material and not only its total amount.

For the sake of generality all the results contained in the paper are obtained in the case of \emph{heterogeneous cost functions} \(c=c(x,t)\), which may depend also on the spatial variable \(x\in\Om\) (see Subsection \ref{sec:MOPc} for a deeper insight on results and relative hypotheses). However, in this introductory chapter we prefer to present the results in the {\it homogeneous} case $c=c(t)$, that is of particular interest and allows a simpler presentation.

Under the above hypotheses on the cost function \(c\), the functional \(\mathcal C\) turns out to be convex and lower semicontinuous with respect to the \(w^*\)-topology on \(\mathscr M(\overline\Omega)\). Thus, whenever the distribution \(f\) is such that the domain of the 
functional \(\mu\mapsto -\mathcal E_f(\mu)+\mathcal C(\mu)\) is non-empty, we obtain, by means of Direct Methods in the Calculus of Variations, the existence of an optimal measure $\mu_{\rm opt}$ for \eqref{eq:MOPc} (see Subsection \ref{sec:Existence}, Theorem \ref{thm:existence_mu_opt}).

Our goal is the characterization of the optimal measure $\mu_{\rm opt}$ in terms of some suitable {\it auxiliary variational problem}. Therefore, we introduce the problem
\begin{equation}\tag{\({\rm AP}_{f,c}\)}\label{eq:aux_pb_2}
{\rm I}_{f,c} \coloneqq \inf_{u\in H^{1}_0(\Omega)} \int_{\Omega} 
c^*\left(\frac{|\nabla u|^2}{2}\right)\,\rmd x 
-\langle f,u\rangle.
\end{equation}
The notation \(c^*\) above stands for the conjugate function of the cost function \(c\) (cf.\ Subsection 
\ref{sec:Duality}). 
In particular, in the case \(c(t)=t/2\) for \(t\ge0\) and \(+\infty\) 
otherwise, we have that the quantity ${\rm I}_{f,c}$ above coincides with ${\rm I}_f$ and the conditions 
obtained in Theorem \ref{casoL} below coincide with the Monge-Kantorovich PDE in 
\cite{BB2001}.

For the sake of clarity we state here our main results, in the particular case of homogeneous cost functions $c(t)$, distinguishing 
two cases: the first one, called \emph{superlinear case} or briefly {\rm (SL)}, is concerned with cost functions \(c\) such that 
\(c^\infty(1)=+\infty\), while the second one, called \emph{linear case} or briefly {\rm (L)}, is concerned with cost functions \(c\) such that \(c^\infty(1)<+\infty\) (i.e. that have a linear growth at infinity). Note that, in the superlinear case, the competing measures $\mu$ are necessarily absolutely continuous with respect to the Lebesgue measure  (otherwise $\mathcal C(\mu)=+\infty$), therefore in the sequel we often refer to them by considering simply the density of their absolutely continuous part $a(x)\in L^1(\Om)$.

\begin{thm}[Existence and optimality conditions in the superlinear case]\label{casoSL}
Assume that  \(\Omega\) is a bounded Lipschitz domain and that \(c\) is a homogeneous cost function with $c^\infty(1)=+\infty$. Then for every \(f\in H^{-1}(\Om)\) any solution $\mu_{\rm opt}$ of problem \eqref{eq:MOPc} is absolutely continuous with respect to the Lebesgue measure. In addition, the auxiliary variational problem \eqref{eq:aux_pb_2} admits a solution $\bar u$ and there exists an optimal measure $\mu_{\rm opt}=a_{\rm opt}\,\rmd x$ such that the couple $(a_{\rm opt},\bar u)\in L^1(\Omega)\times H^1_0(\Omega)$ is identified by the differential conditions:
\begin{equation}\label{eq:OC_SL_intro}
\begin{cases}
-{\rm div}(a_{\rm opt}\nabla\bar u)=f\,\text{ in }\mathscr D'(\Omega);\\
a_{\rm opt}(x)\in \partial c^*\left(\frac{|\nabla \bar u(x)|^2}{2}\right)\,\text{ for }\mathcal L^n\text{-a.e.\ }x\in\Omega.
\end{cases}
\end{equation}
\end{thm}

In other words, in the superlinear case, in order to recover an optimal density $a_{\rm opt}$, we first solve the auxiliary variational problem \eqref{eq:aux_pb_2} obtaining the corresponding solution $\bar u$; then it is enough to take $a_{\rm opt}(x)\in \partial c^*\left(\frac{|\nabla \bar u(x)|^2}{2}\right)$. Note that \(\nabla \bar u\) solves the dual problem associated to \(\mathcal E_{f}(a_{\rm opt})\) via the standard duality arguments  (cf.\ \eqref{eq:dual_problem_SL}). In the superlinear case, Theorem \ref{casoSL} holds even for more general $f$, depending on the growth properties of the cost $c$; in Example \ref{ex:c_quadratic} we show a case in which this occurs. See also Remark \ref{rmk:admissible_f_SL} for a related discussion.

In the linear case instead the optimal measure $\mu_{\rm opt}$ may have singular parts and the corresponding optimality conditions have to take this fact into account.

\begin{thm}[Existence and optimality conditions in the linear case]\label{casoL}
Assume that \(\Omega\) is a bounded Lipschitz domain and that \(c\) is a homogeneous cost function with $c^\infty(1)<+\infty$. Then for every \(f\in \mathscr M(\overline\Omega)\) the auxiliary variational problem \eqref{eq:aux_pb_2} admits a solution $\bar u$ in the space
$${\rm LIP}_{0,\bfc}(\Om)=\big\{u\in {\rm LIP}(\R^n):\,u=0\hbox{ on }\partial\Om,\ \|\nabla u\|_\infty\le \mathbf{c}\big\}\quad \text{ with }\quad \mathbf{c}=\sqrt{2c^\infty(1)}.$$
In addition, the couple $(\mu_{\rm opt},\bar u)\in\mathscr M^+(\overline\Omega)\times{\rm LIP}_{0,\bfc}(\Omega)$ is identified by  the following differential conditions:
\begin{equation}\label{eq:OCLintro}
\begin{split}
1)\quad &-{\rm div}(\mu_{\rm opt}\nabla_{\mu_{\rm opt}}\bar u)
=f\,\text{ in } \mathscr D'(\Omega);\\
2)\quad &\frac{|\nabla_{\mu_{\rm opt}}\bar u(x)|^2}{2}a_{\rm opt}(x)
=c^*\left(\frac{|\nabla\bar u(x)|^2}{2}\right)+
c(a_{\rm opt}(x))\quad\text{holds } 
a_{\rm opt}\mathcal L^n\text{-a.e.\ } x\in \Omega;\\
3) \quad &\frac{|\nabla_{\mu_{\rm opt}} \bar u(x)|^2}{2}= 
c^\infty(1)\,\text{ holds for }\mu^s_{\rm opt}\text{-a.e.\ }x\in \Omega;\\
4)\quad & \mu_{\rm opt}(\partial \Omega)=0.
\end{split}
\end{equation}
\end{thm}

Also in the linear case, the first step to be done, in order to characterize the optimal measure $\mu_{\rm opt}$, is obtaining a solution $\bar u$ to the auxiliary variational problem \eqref{eq:aux_pb_2}; then conditions 1)--4) above make the link between $\bar u$ and $\mu_{\rm opt}$. Note that, due to the linear growth of the cost function $c$, we have $c^*=+\infty$ outside a bounded set, which implies that $\bar u$ is necessarily Lipschitz continuous.

\smallskip

In the last section of this paper (Section \ref{sec:Examples}) we provide some examples and discuss possible variants of the problem.
\section[Preliminaries]{Preliminaries}\label{sec:preliminaries}
Throughout the paper we shall (mainly) use the following notation: the space of all finite
real-valued Borel measures on \(\R^n\) will be denoted by  
\(\mathscr M(\R^n)\), while \(\mathscr M^+(\R^n)\) stands for the 
set of all non-negative \(\mu\in \mathscr M(\R^n)\). Given any compact set \(K\subset \R^n\), we set
\[
\begin{split}
\mathscr M(K)\coloneqq &\{\mu\in \mathscr M(\R^n):\, {\rm spt}(\mu)\subset K\},\\
\mathscr M^+(K)\coloneqq &\{\mu\in\mathscr M(K):\, \mu\text{ non-negative}\}.
\end{split}
\]
Given any \(\mu\in \mathscr M^+(\R^n)\) and \(p\in [1,+\infty)\), we denote by 
\(L^p_\mu(\Omega,\R^n)\) the space of \(p\)-integrable maps from \(\Omega\) into \(\R^n\); in the case of 
real valued functions we simply write \(L^p_\mu(\Omega)\). We denote by \(\mathcal L^n\) the 
\(n\)-dimensional Lebesgue measure in \(\R^n\). In the case of \(\mu=\mathcal L^n\) we omit the subscript 
\(\mu\) in the notation and we simply write \(L^p(\Omega, \R^n)\). In the integration of functions, we use 
equivalently the notations \(\int_\Omega u(x)\,\rmd\mathcal L^n(x)\) and \(\int_{\Omega}u(x)\, \rmd x\).

The space of smooth functions with compact support in \(\Omega\) is denoted by \(\mathscr D(\Omega)\). The 
notation \(\mathscr D'(\Omega)\) stands for the space of distributions on \(\Omega\), while we denote by 
\(\mathscr D'(\overline\Omega)\) the subset of \(\mathscr D'(\R^n)\) of distributions whose support is 
contained in \(\overline\Omega\). The space of continuous functions on a compact set \(K\subseteq \R^n\) is 
denoted by \(\mathscr C(K)\), while \(\|\cdot\|_\infty\) stands for the maximum norm. 
\subsection{Convex functionals on the space of measures}\label{sec:FunctionalsOnMeasures}
In this subsection we recall some terminology and results regarding the lower semicontinuity of convex 
functionals defined on the space of measures.

Let \((\mathrm X, \mathscr B(\mathrm X),\mathbf m)\) be a Borel space, with \(\mathrm X\) a separable, 
locally compact metric space, \(\mathscr B({\rm X})\) a Borel \(\sigma\)-algebra on \({\rm X}\) and 
\(\mathbf m\) a non-negative, non-atomic and finite measure on \(\mathrm X\). A function \(\varphi\colon 
\mathrm X\times \R\to\R\cup\{+\infty\}\) is said to be an \emph{integrand} if it is Borel measurable. An 
integrand \(\varphi\) is said to be  \emph{normal} if  \(\varphi(x,\cdot)\) 
is lower semicontinuous for \(\mathbf m\)-a.e.\ \(x\in \mathrm X\), while it is said to be  \emph{convex} if 
\(\varphi(x,\cdot)\)  is convex for \(\mathbf m\)-a.e.\ \(x\in \mathrm X\).

Given a  convex function \(g:\R\to \R\cup\{+\infty\}\), we define the recession function \(g^{\infty}:\R\to 
\R\cup \{+\infty\}\) of \(g\) as
\begin{equation}\label{eq:recession}
g^{\infty}(t)=\lim_{s\to +\infty}\frac{g(t_0+st)}{s}, \quad \text{ for every }t\in \R,
\end{equation}
where \(t_0\) is an arbitrary point in \({\rm Dom}(g)\neq \emptyset\) (for the definition of proper function 
and its domain see Subsection \ref{sec:Duality} below). The following properties of the recession function 
\(g^\infty\) can be easily deduced:
\begin{itemize}\label{item:recession}
\item [1)] \(g^\infty\) is positively one-homogeneous, i.e. \(g^\infty(\lambda\,t)=\lambda\,g^\infty(t)\) for every \(\lambda>0\) and all \(t\in \R\);
\item [2)] given any \(t_0\in {\rm Dom}(g)\) it holds that \(g(t)\leq g^\infty(1)(t-t_0)\) for all \(t\in\R\).
\end{itemize}
\smallskip

Below, we use the notation \(\mathscr M(\rm X)\) for the space of all real-valued measures with finite 
variation on \(\rm X\),
while \(\mathscr C_0(\rm X)\) stands for 
the space of continuous functions vanishing on the boundary: namely, those continuous function \(g\colon \rm 
X\to \R\) such that for every \(\varepsilon>0\) there exists \(K_{\varepsilon}\subseteq \rm X\) compact for 
which \(|g(x)|<\varepsilon\) holds everywhere on \({\rm X}\setminus K_{\varepsilon}\). We shall denote by 
\(\|\cdot\|_{\sf TV}\) the total variation norm on the space \(\mathscr M(\rm X)\). The space \(\mathscr 
M({\rm X})\) can be also characterized as the  dual of 
the Banach space \((\mathscr C_0(\rm X), \|\cdot\|_{\infty})\), where the duality paring between \(\mu\in 
\mathscr M(\rm X)\) 
and \(g\in \mathscr C_0(\rm X)\) is given by 
\[
\langle \mu,g \rangle\coloneqq \int g\, \rmd \mu.
\]
We shall often consider the \(\rm weak^*\)-topology on \(\mathscr M(\rm X)\); recall that 
a sequence \((\mu_i)_{i\in \N}\subseteq \mathscr M(\rm X)\) \(\rm weakly^*\)-converges to 
\(\mu\in \mathscr M(\rm X)\) if for every \(g\in \mathscr C_0(\rm X)\) it holds that
\[
\int g\,\rmd\mu_i\to \int g\,\rmd \mu\quad \text{ as }\,i\to +\infty.
\]

Let us also recall that for any \(\mu\in \mathscr M(\rm X)\) there exists a unique 
1-integrable (with respcet to \(\bf m\)) function \(a\in L^1_{\bf m}(\rm X)\) and 
a unique measure \(\mu^s\in \mathscr M(\rm X)\) singular with respect to \(\bf m\), 
such that 
\begin{equation}\label{eq:decomposition_measures}
\mu= a{\bf m}+\mu^s.
\end{equation}

In what follows our interest is in the convex functional \(\Phi\colon \mathscr M(\rm X)\to [0,+\infty]\) 
given by
\begin{equation}\label{eq:Phi}
\Phi(\mu)\coloneqq \int \varphi(x,a(x))\,\rmd {\bf m}(x)
+\int\varphi^\infty\bigg(x,\frac{\rmd\mu^s}{\rmd |\mu^s|}\bigg)\,\rmd |\mu^s|(x),\quad 
\text{ for every }\mu\in \mathscr M(\mathrm X),
\end{equation}
where \(\varphi\colon \rm X\times \R\to [0,+\infty]\) is a normal convex integrand and 
\(\varphi^\infty(x,\cdot)\coloneqq \big(\varphi(x,\cdot)\big)^\infty\) for every \(x\in {\rm X}\). In the 
particular case when \(\varphi\) is independent of the variable \(x\), the functional \(\Phi\) is nothing but the lower semicontinuous envelope {\color{blue}} in the weak*-topology on \(\mathscr M(\rm X)\) of its restriction to the space \(L^1_{\bf m}(\rm X)\), which provides its weak* lower semicontinuity (see \cite{AB}). This is not always the case when we deal with an \(x\)-dependent integrand \(\varphi\); see for instance \cite{AB}, \cite{BV}, \cite{DGBDM}, and \cite{Olech} for the integral representation of the lower semicontinuous envelope \(\bar\Phi\) of the functional \(\Phi|_{L^1_{\bf m}}\), 
as well as for some sufficient conditions granting the equality \(\Phi=\bar \Phi\). 

Nevertheless, motivated by the \(x\)-independent case, we shall concentrate on the functionals \(\Phi\) of the above form, which coincide with the lower semicontinuous envelope of \(\Phi|_{L^1(\bf m)}\). Sufficient conditions one may impose on the integrand \(\varphi\) to achieve such a property are given by the following theorem, proven in \cite[Theorem 8]{BV}.

\begin{thm}\label{thm:semicontinuity_of _integral_funct_on _measures}
Let \((\rm X, \mathscr B(\rm X),\mathbf m)\) be a Borel space and let \(\varphi\colon \rm X\times \R\to 
[0,+\infty]\) be a  normal convex integrand satisfying the following condition:
\[
\text{the functions }\,\varphi^*(\cdot, t)\,\text{ and }\,\varphi^\infty(\cdot, t) \text{ are upper semicontinuous for all } t\in \R.
\]
Then the convex functional \(\Phi\colon \mathscr M(\mathrm X)\to [0,+\infty]\) given by \eqref{eq:Phi} coincides with the lower semicontinuous envelope in the \(weak^*\)-topology on \(\mathscr M(\rm X)\) of the functional \(\Phi|_{L^1(\bf m)}\).
\end{thm}

\subsection{Conjugate functions}\label{sec:Duality}
Let \(V\) be a normed vector space and denote by \(V^*\) its topological dual. 
In what follows, we consider \(V\) and \(V^*\) endowed with
weak and  \({\rm weak}^*\) topology, respectively, which make them 
Hausdorff locally convex topological spaces. 
We denote by  \(\langle\cdot,\cdot\rangle\) bilinear pairing between \(V\) and \(V^*\).
Given a function \(F\colon V\to\R\cup\{+\infty\}\), we set
\[{\rm Dom}(F)\coloneqq \{v\in V:\,F(v)<+\infty\},\]
and call it the \emph{effective domain} of \(F\). A function \(F\) is said to be \emph{proper} if \({\rm 
Dom}(F)\neq \emptyset\).
Given a proper function \(F\), the \emph{Fenchel conjugate} (or briefly, \emph{conjugate}) of \(F\) is the 
function \(F^*\colon V^*\to\R\cup\{+\infty\}\) given by
\begin{equation}\label{eq:polar_function}
F^*(v^*)\coloneqq \sup_{v\in V}\big\{\langle v,v^* \rangle-F(v)\big\}
=\sup_{v\in {\rm Dom}(F)}\big\{\langle v,v^* \rangle-F(v)\big\},\quad \text{ for every }v^*\in V^*.
\end{equation}
We list below some useful properties of conjugate functions: fix any
\(F,G\colon V\to \R \cup \{+\infty\}\). Then we have
\begin{itemize}
\item [1)] \(F^*\) is convex and lower semicontinuous with respect to the 
\({\rm weak}^*\) topology on \(V^*\);
\item [2)] if \(F\leq G\) then \(F^*\geq G^*\);
\item [3)] for every \((v,v^*)\in V\times V^*\) it holds that
\[
\langle v, v^*\rangle \leq F(v)+F^*(v^*).
\]
\end{itemize}
The \emph{subdifferential} of \(F\) at \(v\) is defined as
$$\partial F(v)\coloneqq\{v^*\in V^*:\,\langle v, v^*\rangle=F(v)+F^*(v^*)\}.$$
In the next lemma we recall a well-known formula for the subdifferential of a composite function.

\begin{lemma}\label{lem:subdif_composition}
Let \(h\colon \R\to \R\) be a non-decreasing
convex function and let \(G\colon 
R^n\to \R\) be continuous convex function. Then the composite function 
\(\varphi\coloneqq h\circ G\) is convex and it holds that 
\[
\partial \varphi(z)=\partial h\big(G(z)\big)\cdot \partial G(z),\quad \text{ for all }z\in \R^n,
\]
where by the product in the right-hand side we mean the algebraic product of two sets.
\end{lemma}

\begin{thm}
[{\cite[Chapter IX, Proposition 1.2 and Proposition 2.1]{Ekeland}}]\label{thm:Ekeland_integral_conjugate}
Let \(A\subseteq \R^n\) be an open bounded set
and let \(\varphi\colon A\times \R\to [0,+\infty]\) be a normal convex integrand.
Then the function 
\(\varphi^*\colon A\times \R\to \overline \R\) defined as 
\[
\varphi^*(x,t)\coloneqq \sup_{s\in \R} s\cdot t-\varphi(x,s)\, \quad \text{ for all }(x,t)\in A\times \R
\]
is a normal convex integrand as well. Moreover, if we define 
the functional \(\Phi\colon L^1(A)\to [0,+\infty]\) as
\[
\Phi(g)\coloneqq \int_{A}\varphi\big(x,g(x)\big)\,\rmd \mathcal L^n(x), 
\quad \text{ for every } g\in L^1(A)
\]
and assume that there exists \(g_0\in L^\infty(A)\) such that \(\Phi(g_0)<+\infty\),
then we have
\begin{equation}\label{eq:conjugate_integral_functional}
\Phi^*(h)=\int_{A}\varphi^*\big(x,h(x)\big)\,\rmd \mathcal L^n(x), 
\quad \text{ for every } h\in L^{\infty}(A).
\end{equation}
\end{thm}

We recall here a result from min/max theory that is useful for our purposes 
(for its proof see for instance \cite{Clarke} or \cite{Sorin}).

\begin{thm}\label{thm:min_max_Sorin}
Let \(V,W\) be two topological vector spaces, let \(K\subseteq V\) be a compact and convex set, and let 
\(C\subseteq W\) be a convex set. Suppose that the function \(L\colon K\times C\to \R\) satisfies
\begin{itemize}
\item [a)] for each \(v\in K\), the function \(L(v,\cdot)\colon C\to \R\) is convex,
\item [b)] for each \(w\in C\), the function \(L(\cdot, w)\colon K\to \R\) is upper semicontinuous and 
concave.
\end{itemize}
Then we have that 
\begin{equation}
\sup_{v\in K}\inf_{w\in C}L(v,w)=\inf_{w\in C}\sup_{v\in K}L(v,w).
\end{equation}
\end{thm}
The following theorem gives a relation between primal and dual problems in the case of compositions with linear operators.
For its proof we refer to \cite[Proposition 2.5]{Bouchite_Convex}.

\begin{thm}\label{thm:dual_problem_general}
Let \(U\) and \(V\) be two Banach spaces and let \(A\colon U\to V\) be a linear operator with the dense domain denoted \(D(A)\). 
Let \(\Psi\colon V\to \R\cup \{+\infty\}\) be convex and lower semicontinuous function. Let \(F\colon U\to \R\cup\{+\infty\}\) be given by
\[F(u)\coloneqq\begin{cases}
 \Psi(Au)&\text{if } u\in D(A)\\
 +\infty&\text{otherwise.} 
\end{cases}\]
Assume that there exists \(u_0\in D(A)\) such that \(F(u_0)<+\infty\) and that \(\Psi\) is continuous at \(Au_0\). Let \(\psi\colon U\to \R\cup\{+\infty\}\) be a convex function. 
Then
\[
\inf_{u\in U} \Psi(Au)+\psi(u)=
\sup_{v\in V^*} -\psi^*(-A^*v)-\Psi^*(v).
\]
Furthermore, we have that 
\begin{equation}\label{eq:optimal_pair}
\text{ a pair }(u,v)\in U\times V^* \text{ is optimal}\ \iff\ v\in \partial \Psi(Au) \text{ and } -A^*v\in \partial \psi(u).  
\end{equation}
\end{thm}

\section[Mass optimization problem]{Mass optimization problem with convex cost}
\subsection[Introduction to the problem]{Introduction to the problem}\label{sec:MOPc}
From now on, we fix a bounded domain (i.e.\ open, bounded and connected set) \(\Omega\subseteq \R^n\) with 
Lipschitz boundary, whose closure \(\overline\Omega\) represents a given \emph{design region}. Given \(f\in 
\mathscr D'(\overline\Omega)\), representing a heat source density, the total energy of the system 
associated with a distribution of a given conductor \(\mu\in\mathscr M^+(\overline\Omega)\) and a smooth 
temperature profile \(u\in \mathscr D(\Omega)\) is given by the quantity
\begin{equation}
\mathcal F_f(\mu, u)\coloneqq \frac{1}{2}\int|\nabla u|^2\,\rmd\mu-\langle f, u\rangle.
\end{equation}
We consider the energy functional \(\mathcal E_f\colon \mathscr M(\overline\Omega)\to\overline\R\) given by 
\begin{equation}
\mathcal E_f(\mu)\coloneqq\left\{\begin{aligned}
&\inf\Big\{\mathcal F_f(\mu, u):\, u\in \mathscr D(\Omega)\Big\},\\
&-\infty,
\end{aligned}
\begin{aligned}
&\quad\mbox{ for }\, \mu\in \mathscr M^+(\overline\Omega),\\
&\quad\mbox{ otherwise.}
\end{aligned}\right.
\end{equation}
Namely, the quantity \(\mathcal E_f(\mu)\) will be considered as the energy associated with the distribution 
\(\mu\in\mathscr M^+(\overline\Omega)\) of a given conductor.

We are interested in finding the best distribution of a given conductor, in order to achieve 
the \emph{minimal compliance}, under the presence of a \emph{convex penalization term} (that we shall also 
refer to as \emph{convex cost}). More precisely, we define the \emph{cost functional} 
\(\mathcal C\colon \mathscr M(\overline\Omega)\to [0,+\infty]\)
\begin{equation}\label{eq:cost_functional}
\mathcal C(\mu)\coloneqq \int_{\Omega} c(x,a(x))\,\rmd x
+\int c^\infty\bigg(x,\frac{\rmd\mu^s}{\rmd |\mu^s|}\bigg)\,\rmd |\mu^s|(x),\quad 
\text{ for every }\mu=a\,\mathcal L^n+\mu^s\in \mathscr M(\overline \Omega),
\end{equation}
where the \emph{(heterogeneous) cost function} \(c\colon\overline \Omega\times \R\to [0,+\infty]\) is a 
normal convex integrand (in the Borel space \((\overline\Omega, \mathscr B(\overline\Omega), \mathcal L^n)\)), 
satisfying the following assumptions.
\begin{itemize}
\item [(P1)]It holds that
\begin{equation}\label{eq:c}
\begin{split}
c(x,t)\geq \alpha\, t+\beta,&\quad \text{ for all }x\in \overline \Omega \text{ and all }t\geq 0, \text{ for some }\alpha>0 \text{ and }\beta \in \R,\\
c(x,t)=+\infty,&\quad \text{ for all }x\in \overline \Omega \text{ and all }t<0.
\end{split}
\end{equation}
\item [(P2)] There exists \(a_0\in L^{\infty}(\Omega)\) and \(\lambda_0>0\)
such that 
\[a_0(x)\geq \lambda_0 \text{ for } \mathcal L^n\text{-a.e.\ } x\in \Omega\quad
\text{ and }\quad
c(\cdot, a_0(\cdot))\in L^1(\Omega).\]
\item [(P3)] the functions \(c^*(\cdot, t)\)
and 
\(c^\infty(\cdot, t)\) are upper semicontinuous for all \(t\in \R.\)
\end{itemize}

For any cost function \(c\) that is \(x\)-independent the hypotheses (P1) and (P2) are exactly 
the ones given in \eqref{propc} and \eqref{tzero}, respectively, while the property (P3) is redundant.
These cost functions motivate our investigation (cf.\ Introduction)  and we will refer to them as  
\emph{homogeneous cost functions}. 

\begin{remark}\label{rmk:P1-P3}
A couple of comments about the hypothesis on the cost functions are in order. The properties (P1) and (P3) are needed for the existence result for our mass optimization problem: (P1) is giving us compactness property, while (P3) provides us with the lower semicontinuity of the functional \(\mathcal C\) (with respect to the \(w^*\)-topology on \(\mathscr M(\overline\Omega)\)), due to
Theorem \ref{thm:semicontinuity_of _integral_funct_on _measures}. In particular, \(\mathcal C\) is the lower semicontinuous envelope in the \(w^*\)-topology on \(\mathscr M(\overline\Omega)\)) of the functional \(\mathcal C|_{L^1(\Omega)}\). 
The property (P2) plays a role in making the relation of our minimization problem with the auxiliary problem \eqref{eq:aux_pb_2} (homogeneous case) and \eqref{eq:aux_problem} (the general, heterogeneous case),
where a \(\Gamma\)-convergence argument will be required (see Theorem \ref{thm:aux_related_to_optimal_mu}).
\end{remark}

Having fixed a cost function \(c\) as above, and having fixed any distribution
\(f\in \mathscr D'(\overline\Omega)\),
we define the functional 
\(\mathcal J_{c,f}\colon \mathscr M(\overline\Omega)\to [0,+\infty]\)
by
\begin{equation}
\mathcal J_{c,f}(\mu)\coloneqq
-\mathcal E_f(\mu)+\mathcal C(\mu), 
\quad \text{ for every }\mu\in \mathscr M(\overline \Omega).
\end{equation}
Notice that, depending on the cost function \(c\), for some heat sources \(f\) it might happen that 
\(\mathcal J_{c,f}\equiv +\infty\). In order to avoid these degenerate cases, we define the set of 
\(c\)-\emph{admissible sources}
\begin{equation}
{\rm Adm}_c\coloneqq \{f\in \mathscr D'(\overline\Omega):\, 
{\rm Dom}(\mathcal J_{c,f})\neq \emptyset\}.
\end{equation}

\begin{remark}\label{rmk:adm}
It is not difficult to see that, due to the property \textrm{(P2)} of the cost function \(c\), it holds
that \(H^{-1}(\Omega)\subseteq {\rm Adm}_c\). In the case in which \(c\) has the property 
\(\sup_{x\in \overline\Omega} c^\infty(x,1)<+\infty\) (that we shall refer to as the
linear (or briefly, (L)) case),
we will show in Section \ref{sec:linear} that \(\mathscr M(\overline\Omega)\subseteq {\rm Adm}_c\). See 
Theorem \ref{thm:aux_related_to_optimal_mu_linear}.
\end{remark}

Thus, given a cost function \(c\) as above and given \(f\in {\rm Adm}_c\), our aim is to study the following 
minimization problem:
\begin{equation}\label{eq:mass_opt_pb}
\underset{\mu\in \mathscr M^+(\overline\Omega)}
{\inf}\mathcal J_{c,f}(\mu).
\end{equation}

\subsection{The existence of minimizers}\label{sec:Existence} 

\begin{thm}\label{thm:existence_mu_opt}
Let \(f\in {\rm Adm}_c\). Then, the set 
\begin{equation}\label{eq:optimal_mu}
\mathscr M_{\rm opt}\coloneqq \{\mu_{\rm opt}\in \mathscr M^+(\overline \Omega):\,\mathcal 
J_{c,f}(\mu_{\rm opt})=\inf_{\mu\in \mathscr M^+(\overline \Omega)}\mathcal J_{c,f}(\mu)\}
\end{equation}
is nonempty. Moreover, there exists \(\mathbf{k}>0\) such that for every \(\mu_{\rm opt}\in \mathscr M_{\rm 
opt}\) and every \(k\geq \mathbf{k}\) we have
\begin{equation}\label{eq:sup_inf_exchange}
\inf_{u\in \mathscr D(\Omega)}\mathcal F_f(\mu_{\rm opt}, u)-\mathcal C(\mu_{\rm opt})
=\inf_{u\in \mathscr D(\Omega)}\, \sup_{\mu\in \mathscr K^+_k(\overline \Omega)}
\,\mathcal F_f(\mu, u)-\mathcal C(\mu)
\end{equation}
where we set  
\[\mathscr K^+_k(\overline \Omega)\coloneqq \big\{\mu\in \mathscr M^+(\overline \Omega):\, 
\mu(\overline\Omega)\leq k\big\},\quad \text{ for every } k\in \N.\]
\end{thm}

\begin{proof}
The functional \(-\mathcal E_f\) is convex and \(w^*\)-lower semicontinuous, being the supremum of 
the linear and \(w^*\)-continuous functionals \(-\mathcal F_f(\cdot,u)\), with \(u\in \mathscr D(\Omega)\). 
By Theorem \ref{thm:semicontinuity_of _integral_funct_on _measures}, 
we also have that the convex functional 
\(\mathcal C\) is \(w^*\)-lower semicontinuous. 
Thus the functional \(\mathcal J_{c,f}=-\mathcal E_f+\mathcal C\) is convex and \(w^*\)-lower semicontinuos. 
Further, due to the growth condition of the cost function \(c\) (see \eqref{eq:c}), we have
\[\mathcal C(\mu)\ge\mu(\overline\Omega),\qquad\text{for all }\mu\in\mathscr M^+(\overline\Omega).\]
This, together with the fact that \(f\in{\rm Adm}_c\), implies that there exists \(\mathbf k>0\) such that
\[
\emptyset\neq \big\{\mu\in \mathscr M^+(\overline \Omega):\, 
\mathcal J_{c,f}(\mu)\leq \mathbf k\big\}\subseteq 
\mathscr K^+_{\mathbf k}(\overline \Omega).
\]
Being the set \(\mathscr K^+_{\mathbf k}(\overline \Omega)\) \(w^*\)-compact, the set 
\(\big\{\mu\in\mathscr M^+(\overline\Omega):\,\mathcal J_{c,f}(\mu)\le\mathbf{k}\big\}\)
(which is \(w^*\)-closed due to the \(w^*\)-lower semicontinuity of \(\mathcal J_{c,f}\))
is \(w^*\)-compact as well. By the direct method of the Calculus of Variations, we conclude that
\(\mathscr M_{\rm opt}\ne\emptyset\).

In order to show \eqref{eq:sup_inf_exchange}, we observe that for any \(k\ge\mathbf{k}\) the set 
\(\mathscr K^+_k(\overline\Omega)\) is nonempty, and
\begin{equation}\label{eq:aux_0}
\inf_{\mu\in \mathscr M^+(\overline\Omega)}\mathcal J_{c,f}(\mu)
=\inf_{\mu\in \mathscr K^+_{k}(\overline \Omega)}\mathcal J_{c,f}(\mu).
\end{equation}  
For any \(k\ge\mathbf{k}\) our minimization problem \eqref{eq:mass_opt_pb} can be written as a min/max-type 
problem in the following way:
\begin{equation}\label{eq:aux_1}
\begin{split}
\inf_{\mu\in \mathscr K^+_{k}(\overline\Omega)}\mathcal J_{c,f}(\mu)
&=\inf_{\mu\in \mathscr K^+_{k}(\overline\Omega)}
\Big(-\!\!\inf_{u\in \mathscr D(\Omega)}\mathcal F_f(\mu,u)+\mathcal C(\mu)\Big)\\
&=-\!\!\sup_{\mu\in \mathscr K^+_{k}(\overline \Omega)}\,
\inf_{u\in \mathscr D(\Omega)}\mathcal F_f(\mu, u)-\mathcal C(\mu).
\end{split}
\end{equation}
Thus, by applying Theorem \ref{thm:min_max_Sorin} with the compact convex set
\(K=\mathscr K^+_{k}(\overline \Omega)\subseteq \mathscr M(\overline \Omega)\), 
the convex set \(C= \mathscr D(\Omega)\), and the function
\(L\colon K\times C\to \R\) given by \(L(\mu, u):=\mathcal F_f(\mu, u)-\mathcal C(\mu)\),
we obtain
\begin{equation}\label{eq:aux_2}
\sup_{\mu\in \mathscr K^+_{k}(\overline \Omega)}\,
\inf_{u\in \mathscr D(\Omega)}\mathcal F_f(\mu, u)-\mathcal C(\mu)
=\inf_{u\in \mathscr D(\Omega)}\, \sup_{\mu\in \mathscr K^+_{k}(\overline \Omega)}
\,\mathcal F_f(\mu, u)-\mathcal C(\mu).
\end{equation}
Putting together \eqref{eq:aux_0}, \eqref{eq:aux_1} and \eqref{eq:aux_2}, we have that 
\eqref{eq:sup_inf_exchange} holds, concluding the proof.
\end{proof}

Next we prove an auxiliary result
providing us with the formula for 
the conjugate \(\mathcal C^*\) of the functional \(\mathcal C\), that will be needed in the sequel.
\begin{lemma}\label{lem:C_conjugate}
The Fenchel conjugate \(\mathcal C^*\) of the cost functional 
\(\mathcal C\) defined in \eqref{eq:cost_functional} reads as
\[
\mathcal C^*(h)=\int_{\Omega}c^*\big(x,h(x)\big)\,\rmd x,\quad \text{ for all }h\in \mathscr C(\ove\Omega).
\]
\end{lemma}
\begin{proof}
Given any \(h\in \mathscr C(\overline \Omega)\), by the definition of the Fenchel conjugate in 
\eqref{eq:polar_function} and taking into account the fact that \(\mathcal C\) is the lower 
semicontinuous envelope of \(\mathcal C|_{L^1(\Omega)}\) in the weak*-topology on measures (tanks to the property (P3) of the cost function \(c\)), we have that
\[
\mathcal C^*(h)
=\sup_{\mu\in \mathscr M(\overline \Omega)}\int h\,\rmd \mu-\mathcal C(\mu)
=\sup_{a\in L^1(\Omega)}
\int_{\Omega}\big(h(x) a(x)-c\big(x,a(x)\big)\big)\,\rmd x.
\]
By using Theorem \ref{thm:Ekeland_integral_conjugate} applied to the functional
\[
L^1(\Omega)\ni g\mapsto \int_{\Omega} c\big(x,g(x)\big)\,\rmd x,
\]
which is proper, since the functional \(\mathcal C\) enjoys the same property,
we have that
\[
\sup_{a\in L^1(\Omega)}
\int_{\Omega}\big(a(x)h(x)-c\big(x,a(x)\big)\big)\,\rmd x
=\int_{\Omega} c^*\big(x,h(x)\big)\,\rmd x
\]
as required.
\end{proof}

We collect here some observations that will be useful in the next sections.
Let us recall from Theorem \ref{thm:dual_problem_general} that given any 
\(\mu\in \mathscr M^+(\ove\Omega)\),
the dual problem associated to the minimization problem 
\(
\inf_{u\in \mathscr D(\Omega)} 
\mathcal F_f(\mu,u)-\mathcal C(\mu)
\)
reads 
\begin{equation}\label{eq:dual_problem_SL}
\sup \left\{-\int \frac{|\sigma(x)|^2}{2}\,\rmd \mu:\,\sigma \in L^2_\mu(\Omega,\R^n),\, -{\rm div}(\mu\sigma)=f \text{ in }\mathscr D'(\Omega)\right\}-\mathcal C(\mu).
\end{equation}
In the above formula, the notation \({\rm div}(\mu\sigma)\) stands for the  distributional  \(\mu\)-\emph{divergence} of a vector field \(\sigma\). Namely, \({\rm div}(\mu\sigma)\in \mathscr D'(\Omega)\) is defined by
\begin{equation}\label{eq:mu_divergence}
\langle {\rm div}(\mu\sigma) , u\rangle=-\int \sigma\cdot \nabla u\,\rmd \mu, \quad \text{ for every }u\in \mathscr D(\Omega).
\end{equation}
Recall also that, whenever \(f\in {\rm Adm}_c\) and \(\mu\in {\rm Dom}(\mathcal J_{f,c})\), we have that
the supremum in \eqref{eq:dual_problem_SL} is achieved.

\section{The case of superlinear cost function}
In this section we provide optimality conditions for the optimal conductivities \(\mu_{\rm opt}\in\mathscr M_{\rm opt}\) when the cost function \(c\) has a superlinear growth at infinity. More specifically, we consider throughout this section the following case:
\begin{equation}\tag{SL}
\textbf{Superlinear case:}\qquad \inf_{x\in \ove\Omega}c^\infty(x,1)\equiv +\infty, \quad \text{ where } c^{\infty}(x,1)\coloneqq \big(c(x,\cdot)\big)^{\infty}(1).
\end{equation}
Notice that in this case the domain of finiteness \({\rm Dom}(\mathcal C)\) of the cost functional 
\(\mathcal C\) is contained in the space of measures that are absolutely continuous with respect to the Lebesgue measure and thus all the optimal 
conductivities \(\mu_{\rm opt}\in \mathscr M_{\rm opt}\) are of the form 
\[
\mu_{\rm opt}=a_{\rm opt}\,\mathcal L^n,\quad \text{ with }a_{\rm opt}\in L^1(\Omega).
\]

For a simpler and clearer presentation, we consider first the homogeneous case, i.e.\ with \(c(x,t)=c(t)\), and then the heterogeneous case, emphasizing in this way the differences due to the \(x\)-dependence of the cost function \(c\).

\subsection{Homogeneous cost function}
Throughout this subsection we add the following assumption to the cost function \(c\):
\[
\textbf{Homogeneous cost:}\qquad c(x,t)=c(t)\, \text{ for all }t\in \R.
\]

In order to characterize the optimal conductors 
\(a_{\rm opt}\) in the next sections we shall examine  the
auxiliary variational problem \eqref{eq:aux_pb_2}. Recall for the future reference that the notation 
\({\rm I}_{f,c}\) stand for the value of the infimum in the problem \eqref{eq:aux_pb_2}.

\begin{thm}[The auxiliary problem in (SL) homogenoeus case]\label{thm:auxiliary_problem}
Let \(f\in H^{-1}(\Omega)\). 
Then it holds that 
\begin{equation}\label{eq:aux_problem_relaxed}
\inf_{u\in \mathscr D(\Omega)}\int_{\Omega} c^*\left(\frac{|\nabla u(x)|^2}{2}\right)\,\rmd x
-\langle f,u\rangle
={\rm I}_{f,c}.
\end{equation}
Moreover, the auxiliary problem \eqref{eq:aux_pb_2}
admits a solution and any solution \(\bar u\in H^1_0(\Omega)\) satisfies 
\begin{equation}\label{eq:optimal_pairs_hom}
-{\rm div}(\bar v)=f,\quad\text{where}\quad\bar v(x)\in\partial c^*\left(\frac{|\nabla \bar u(x)|^2}{2}\right)\nabla\bar u(x)\,\text{ for }\mathcal L^n\text{-a.e.\ }x\in\Omega. 
\end{equation}
\end{thm}

\begin{proof}
The first part of the statement directly follows from \cite[Chapter X, Proposition 2.6]{Ekeland} and the growth condition from below on 
the function 
\[\varphi(s)\coloneqq c^*(s^2/2)\]
granted by the property \eqref{tzero} of the cost function \(c\). 
The same property ensures that the functional 
\[
H^1_0(\Omega)\ni u\mapsto 
{\rm I_{\varphi}}(\nabla u)\coloneqq\int_{\Omega}\varphi\big(\nabla u(x)\big)\,\rmd x
\]
is coercive in the weak topology on \(H^1_0(\Omega)\), thus the problem \eqref{eq:aux_pb_2} admits a solution in \(H^1_0(\Omega)\). The dual problem to \eqref{eq:aux_pb_2} (cf.\ Theorem \ref{thm:dual_problem_general}) reads as 
\[
\sup\big\{ {\rm I}_{\varphi}^*(v):\, 
v\in L^2(\Omega)\,\text{ such that }\, -{\rm div}(v)=f\big\}, 
\]
so that the optimal pairs \((u,v)\) satisfy
\[
-{\rm div}(v)=f\quad \text{ and }\quad
v\in \partial {\rm I}_{\varphi}(\nabla u).
\]
Recalling from Theorem \ref{thm:Ekeland_integral_conjugate}
that \({\rm I}_\varphi^*(v)=\int_{\Omega}\varphi^*(v(x))\,\rmd x\),
the second condition gives
\[
\int_{\Omega}\varphi\big(\nabla u(x)\big)+ \varphi^*\big(v(x)\big)\,\rmd x
= \int_{\Omega} v(x)\cdot \nabla u(x)\,\rmd x,
\]
or in other words that
\(
v(x)\in \partial \varphi(\nabla u(x))\)  for \(\mathcal L^n\)-a.e.\ \(x\in \Omega\).
Being \(c^*\) non-decreasing and \(|\cdot|^2\) a continuous function, it follows from Lemma \ref{lem:subdif_composition} that
\[\partial\varphi(z)
=z\,\partial c^*\left(\frac{|z|^2}{2}\right)\quad\text{ for all }z\in\R^n.\]
Taking \(z=\nabla u(x)\) we obtain the condition stated in \eqref{eq:optimal_pairs_hom}.
\end{proof}

\begin{thm}\label{thm:aux_related_to_optimal_mu}
 Let \(f\in H^{-1}(\Omega)\). Then for every optimal conductivity 
\(a_{\rm opt}\in \mathscr M_{\rm opt}\) it holds that 
\[
\inf_{u\mathscr D(\Omega)}\int_{\Omega} \frac{|\nabla u(x)|^2}{2}a_{\rm opt}(x)\,\rmd x-\langle f,u\rangle -\int_{\Omega} c\big(a_{\rm opt}(x)\big)\,\rmd x=\min_{u\in H^1_0(\Omega)}\int_{\Omega} c^*\left(\frac{|\nabla u(x)|^2}{2}\right)\,\rmd x
-\langle f,u\rangle.
\]
\end{thm}
\begin{proof}
Recall from Remark \ref{rmk:adm} that \(H^{-1}(\Omega)\subseteq {\rm Adm}_c\), 
thus there exists \(\mu_{\rm opt}=a_{\rm opt}\mathcal L^n\in \mathscr M_{\rm opt}\).
\underline{We first show the inequality \(\leq\):} being \(a_{\rm opt}\) an optimal 
conductivity, we have from Theorem \ref{thm:existence_mu_opt}, precisely from \eqref{eq:sup_inf_exchange}, that for \(k\in \N\) big enough it holds
\[
\begin{split}
&\inf_{u\in\mathscr D(\Omega)}\int_{\Omega} \frac{|\nabla u(x)|^2}{2}a_{\rm opt}(x)\,\rmd x-\langle f,u\rangle -\int_{\Omega} c\big(a_{\rm opt}(x)\big)\,\rmd x\\
=&\inf_{u\in \mathscr D(\Omega)}\sup_{\mu\in \mathscr K^+_k(\ove \Omega)}\mathcal F_f(\mu, u)-\mathcal C(\mu)\\
\leq &\inf_{u\in \mathscr D(\Omega)}\sup_{\mu
\in \mathscr M(\ove\Omega)}\mathcal F_f(\mu, u)-\mathcal C(\mu) \\
=&\inf_{u\in \mathscr D(\Omega)}\sup_{a\in
L^1(\Omega)}\int_{\Omega}\frac{|\nabla u(x)|^2}{2}\,\rmd x-\langle f,u\rangle -\int_{\Omega}c\big(a(x)\big)\,\rmd x,
\end{split}
\]
where the last inequality follows from the property of \(\mathcal C\) of being the lower semicontinuous envelope of 
\(\mathcal C|_{L^1(\Omega)}\).
Then, we get the desired inequality by applying first Lemma \ref{lem:C_conjugate} and then Theorem \ref{thm:auxiliary_problem}.

\underline{We now show the inequality \(\geq\):} Let us introduce some
auxiliary notation. 
For every \(k\in \N\) we set
\begin{equation*}
A_k\coloneqq\{t\in \R:\, 0\leq t\leq k/\mathcal L^n(\Omega)\},\quad
\chi_k\coloneqq \chi_{A_k},\quad
c_k\coloneqq c+\chi_k=
\begin{cases}
c(t)&\text{if }t\in A_k,\\
+\infty&\text{otherwise}.
\end{cases}
\end{equation*}
Consider the sequence \((c_k^*)_{k\in \N}\) of functions
\(c_k^*\colon\R\to\overline\R\) given by
\begin{equation}
c_k^*(t)=\sup_{s\in A_k}\,s\cdot t-c_k(s),\quad 
\text{ for all }t\in \R.
\end{equation}
It is straightforward to check that 
\begin{itemize}
\item [\rm (a)] For every \(k\in \N\), the function \(c_k^*\) is a normal convex integrand.
\item [\rm (b)] \((c_k^*)_{k\in \N}\) is an increasing sequence.
\item [\rm (c)] \(c_k^*(t)\nearrow c^*(t)\) 
for every \(t\in \R\).
\item [\rm (d)] Let \(\lambda_0>0\) be as in \eqref{tzero}. Then for every \(\N\ni k\geq \lambda_0\mathcal L^n(\Omega)\) we have
\[c_k^*(t)\geq \lambda_0\, t-c(\lambda_0),\quad \text{ for every } t\in 
\R.\]
\end{itemize}
Since we are in the superlinear case, the right-hand side of the equality in \eqref{eq:sup_inf_exchange} becomes
\[
\inf_{u\in \mathscr D(\Omega)}\sup_{\mu\in \mathscr K_k^+(\overline\Omega)}\,\mathcal F_f(\mu, u)-\mathcal C(\mu)
=\inf_{u\in \mathscr D(\Omega)}\sup_{a\in L^1_k(\Omega)}\int_{\Omega}\frac{|\nabla u(x)|^2}{2}a(x)-c(a(x))\,\rmd x-\langle f, u\rangle.
\]
By noticing that \(L^\infty_{k/\mathcal{L}^n(\Omega)}(\Omega)\subseteq L^1_k(\Omega)\)
and taking \(\N\ni k\geq\lambda_0\mathcal L^n(\Omega)\)
we get the following bound
\[
\begin{split}
\sup_{a\in L^1_k(\Omega)}\int_{\Omega}\frac{|\nabla u(x)|^2}{2}a(x)-c(a(x))\,\rmd x
\geq &
\sup_{a\in L^\infty_{k/\mathcal{L}^n(\Omega)}(\Omega)}\int_{\Omega}\frac{|\nabla 
u(x)|^2}{2}a(x)-c_{k}(a(x))\,\rmd x\\
= & \int_{\Omega} c_{k}^*\left(\frac{|\nabla u(x)|^2}{2}\right)\,\rmd x
\eqqcolon \Phi_{k}(u).
\end{split}
\]
All in all, the above shows that 
\begin{equation}\label{eq:aux_inequality}
\inf_{u\in \mathscr D(\Omega)}\sup_{\mu\in \mathscr K_k^+(\overline\Omega)}\,\mathcal F_f(\mu, u)-\mathcal C(\mu)\geq \inf_{u\in H^1_0(\Omega)}\Phi_k(u)-\langle f,u\rangle.
\end{equation}
We now claim that 
\begin{equation}
\label{eq:aux_equality}
\lim_{k\to +\infty}
\inf_{u\in H^1_0(\Omega)}\Phi_k(u)-\langle f,u\rangle
=\inf_{u\in H^1_0(\Omega)}
\int_{\Omega}c^*\left(\frac{|\nabla u(x)|^2}{2}\right)\,\rmd x
-\langle f,u \rangle.
\end{equation}
Due to the assumptions on \(f\), in the next passage to the \(\Gamma\)-limit
the functional 
\(\langle f,\cdot \rangle\) is treated via continuity, thus we only need to take care about 
the sequence of functionals \(\{\Phi_k\}_{k\in \N}\).
Being \(\{\Phi_k\}_{k\in \N}\) the sequence of 
weak \(H^1_0(\Omega)\)-lower semicontinuous and increasing functionals, as
granted by the properties (a)--(c) of the sequence 
\((c_k^*)_{k\in \N}\) stated above,
its \(\Gamma\)-limit
in the weak \(H^1_0(\Omega)\)-topology coincides with its 
pointwise limit 
(see \cite[Proposition 5.4 and Remark 5.5]{DM}). 
Moreover, since the functionals \(\Phi_k\),
\(k\in \N\) are equicoercive (as the consequence of the property (d) above), it follows from \cite[Theorem 7.8]{DM} that the equality in
\eqref{eq:aux_equality} holds. 
This  together with \eqref{eq:sup_inf_exchange} and \eqref{eq:aux_inequality} concludes the proof.
\end{proof}

\begin{remark}\label{rmk:admissible_f_SL}
We remark  that the above results 
hold even under less restrictive assumptions on 
the admissible distribution \(f\in {\rm Adm}_c\). Namely, it is enough to require that 
\(f\) is continuous with respect to the convergence \(u_i\rightharpoonup u\) in \(H^{1}_0(\Omega)\) weakly 
and \(\sup_i \int_{\Omega}c^*(|\nabla u_i|^2/2)\,\rmd x\) being finite.
\end{remark}

\begin{thm}[Optimality conditions in the (SL) homogeneous case]\label{thm:opt_cond_hom}
Let \(f\in H^{-1}(\Omega)\).
Then, the following are equivalent:
the couple \((a_{\rm opt}, \bar u)\in L^1(\Omega)\times H^1_0(\Omega)\)
satisfies the conditions
(referred to as the \emph{optimality conditions})
\begin{equation}\label{eq:OC_SL}
\begin{split}
1)\quad &-{\rm div}(a_{\rm opt}
\nabla\bar u)=f \, \text{ in }\mathscr D'(\Omega);\\
2)\quad & a_{\rm opt}(x)\in \partial c^*\left(\frac{|\nabla \bar u(x)|^2}{2}\right)\,\text{ for }\mathcal L^n\text{-a.e.\ }x\in \Omega.
\end{split}
\end{equation}
if and only if 
\begin{itemize}
\item [\rm a)] \(a_{\rm opt}\) is a solution of the problem \eqref{eq:MOPc};
\item [\rm b)] \(\bar u\) is a  solution to the problem \eqref{eq:aux_pb_2},
\item [\rm c)] \(\nabla \bar u\) is a solution of the dual problem \eqref{eq:dual_problem_SL} with \(a=a_{\rm opt}\).
\end{itemize}
\end{thm}
\begin{proof}
\underline{Let us prove the implication \((\Rightarrow)\)}.
The conditions 1) and 2) grant that  
\(\nabla\bar u\) is a candidate
in the dual problem \eqref{eq:dual_problem_SL} with \(a=a_{\rm opt}\),
\(\bar u\in H^1_0(\Omega)\) is a 
solution to the problem \eqref{eq:aux_pb_2} (thus condition b) is verified), and the 
vector field \(a_{\rm opt}\nabla \bar u\in L^2(\Omega)\) solves the dual problem to \eqref{eq:aux_pb_2}. 
Taking into account Theorem \ref{thm:auxiliary_problem} and Theorem \ref{thm:aux_related_to_optimal_mu},
the following chain of inequalities holds:
\begin{equation}\label{eq:OC_aux_1}
\begin{split}
&-\int\frac{|\nabla\bar u(x)|^2}{2}a_{\rm opt}(x)\,\rmd x-\int_{\Omega}c\big(a_{\rm opt}(x)\big)\,\rmd x
\leq \inf_{u\in \mathscr D(\Omega)}\mathcal F_f(a_{\rm opt},u)-\mathcal C(a_{\rm opt})\leq \sup_{\mu\in \mathscr M^+(\ove\Omega)}\mathcal J_{f,c}(\mu)\\
&={\rm I}_{f,c}=\int_{\Omega}c^*\left(\frac{|\nabla \bar{u}(x)|^2}{2}\right)\,\rmd x-\langle f, 
\bar u\rangle=\int\frac{|\nabla\bar u(x)|^2}{2} a_{\rm opt}(x)\,\rmd x-\int_{\Omega}c\big(a_{\rm opt}(x)\big)\,
\rmd x-\langle f,\bar u \rangle.
\end{split}
\end{equation}
Take now any sequence \((u_i)_{i\in \N}\subseteq \mathscr D(\Omega)\) such that 
\[u_i\rightharpoonup \bar u \text{ weakly in } H^1_0(\Omega) 
\quad \text{ and }\quad 
\int_{\Omega}c^*\left(\frac{|\nabla \bar u|^2}{2}\right)\,\rmd x
=\lim_{i\to +\infty} \int_{\Omega}c^*\left(\frac{|\nabla  u_i|^2}{2}\right)\,\rmd x,
\]
which exists due to Theorem \ref{thm:auxiliary_problem}.
Up to a non-relabeled subsequence, we also have that
\(\nabla u_i\rightharpoonup w\) weakly in \(L^2_{a_{\rm opt}}(\Omega, \R^n)\) 
for some \(w\in L^2_{a_{\rm opt}}(\Omega, \R^n)\).
Being \(a_{\rm opt}\nabla \bar u\in L^2(\Omega)\) and \(\nabla \bar u\in L^2_{a_{\rm opt}}(\Omega)\), thanks to the condition 2), we have that 
\[
\int_{\Omega} w\cdot \nabla \bar u\, a_{\rm opt}\,\rmd x
=\lim_{i\to +\infty}\int_{\Omega} \nabla u_i \cdot \nabla \bar u \,a_{\rm opt}\,\rmd x
=\int_{\Omega} \nabla \bar u\cdot \nabla \bar u\, a_{\rm opt}\,\rmd x =\int_{\Omega}|\nabla \bar u|^2\,a_{\rm opt}\,\rmd x.
\]
Notice that
\[
\langle f,\bar u \rangle 
=\lim_{i\to +\infty} \langle -{\rm div}(a_{\rm opt}\nabla\bar u),u_i \rangle 
=\lim_{i\to +\infty}\int \nabla\bar u \cdot \nabla u_i \,a_{\rm opt}\,\rmd x
=\int_{\Omega}|\nabla \bar u|^2\,a_{\rm opt}\,\rmd x,\]
which, when plugged  in \eqref{eq:OC_aux_1}, shows that  all the inequalities in \eqref{eq:OC_aux_1} 
equalities. In particular, 
 \(a_{\rm opt}\) solves \eqref{eq:mass_opt_pb} and \(\nabla \bar u\) solves the dual problem \eqref{eq:dual_problem_SL} with \(a=a_{\rm opt}\),
 thus the conditions a) and c) follow. 

\underline{We now prove the implication \((\Leftarrow)\)}. Due to the assumptions and Theorem \ref{thm:aux_related_to_optimal_mu}
we have that 
\begin{equation}\label{eq:OC_aux_2}
-\int\frac{|\nabla\bar u|^2}{2}\, a_{\rm opt}\,\rmd x
-\int_{\Omega}c\big(a_{\rm opt}(x)\big)\,\rmd x
=\inf_{u\in \mathscr D(\Omega)}\mathcal F_f(a_{\rm opt},u)-\mathcal C(a_{\rm opt})
=\int_{\Omega}c^*\left(\frac{|\nabla \bar u|^2}{2}\right)\,\rmd x-\langle f,\bar u\rangle.
\end{equation}
Take now any sequence \((u_i)_{i\in \N}\subseteq \mathscr D(\Omega)\) such that 
\[u_i\rightharpoonup \bar u \text{ weakly in } H^1_0(\Omega) 
\quad \text{ and }\quad 
\int_{\Omega}c^*\left(\frac{|\nabla \bar u|^2}{2}\right)\,\rmd x
=\lim_{i\to +\infty} \int_{\Omega}c^*\left(\frac{|\nabla  u_i|^2}{2}\right)\,\rmd x,
\]
which exists due to Theorem \ref{thm:auxiliary_problem}. 
Up to a non-relabeled subsequence, \(\nabla u_i\) 
weakly converge in \(L^2_{a_{\rm opt}}(\Omega, \R^n)\) to some \( w\).
Notice also that due to the lower semiconitnuity of the norm, we have that
\begin{equation}\label{eq:w_bound}
\int \frac{|w|^2}{2}a_{\rm opt}\,\rmd x
\leq \lim_{i\to +\infty}\int \frac{|\nabla u_i|^2}{2}a_{\rm opt}\,\rmd x
\leq \int_{\Omega}c^*\left(\frac{|\nabla \bar u|^2}{2}\right)\,
\rmd x +\int_{\Omega}c\big(a_{\rm opt}(x)\big)\,\rmd x.
\end{equation}
where the last inequality follows from the conjugate duality and the way in which the sequence \((u_i)_{i\in 
\N}\) has been chosen.
Further, being \(\nabla \bar u\) a solution to \eqref{eq:dual_problem_SL}, we have that 
\[
\langle f,\bar u\rangle =\lim_{i\to +\infty}\langle -{\rm div}(a_{\rm opt}\,\nabla \bar u),u_i\rangle=\lim_{i\to 
+\infty} \int \nabla\bar u \cdot \nabla u_i\, a_{\rm opt}\,\rmd x
=\int \nabla\bar u\cdot w\, a_{\rm opt}\,\rmd x.
\]
Combining the latter equality with \eqref{eq:w_bound} and
\eqref{eq:OC_aux_2} we get that 
\[
\begin{split}
-\int \frac{|\nabla\bar u|^2}{2}a_{\rm opt}\,\rmd x -\int_{\Omega}c\big(a_{\rm opt}(x)\big)\,\rmd x
= &\int_{\Omega}c^*\left(\frac{|\nabla \bar u|^2}{2}\right)\,\rmd x -\int \nabla \bar u\cdot  w\, a_{\rm opt}\,\rmd x\\
\geq&\int \frac{| w|^2}{2}a_{\rm opt}\,\rmd x-\int_{\Omega}c\big(a_{\rm opt}(x)\big)\,\rmd x -\int \nabla \bar u\cdot 
 w\,a_{\rm opt}\,\rmd x.
\end{split}
\]
This implies that \(\nabla\bar u- w=0\) holds \(a_{\rm opt}\mathcal L^n\)-a.e.. Now it follows from \eqref{eq:OC_aux_2} and b) that
\[
\int \frac{|\nabla\bar u|^2}{2}a_{\rm opt}\,\rmd x
=\int_{\Omega}c^*\left(\frac{|\nabla \bar u|^2}{2}\right)\,\rmd x 
+\int_{\Omega}c\big(a_{\rm opt}(x)\big)\,\rmd x,
\]
proving 2). The point 1) is trivially satisfied by the assumption c).
\end{proof}

\subsection{Heterogeneous cost function}\label{sec:Heterogeneous_SL}
We now assume that the cost function \(c\) is dependent also on the domain variable \(x\), namely we consider:

\[
\textbf{Heterogeneous cost:}\qquad c\colon \ove\Omega\times \R\to [0,+\infty]\,\text{ satisfying }{\rm (P1)}-{\rm (P3)}.
\]

As before, we shall look at the auxiliary variational problem,
which in the general setting of the \(x\)-dependent cost reads as:
\begin{equation}\label{eq:aux_problem}
\inf_{u\in H^1_0(\Omega)}\int_{\Omega} c^*\left(x,\frac{|\nabla u(x)|^2}{2}\right)\,\rmd x
-\langle f,u\rangle.
\end{equation}
We will (since no ambiguity occurs) denote by \({\rm I}_{f,c}\) the value of the infimum in \eqref{eq:aux_problem}.
\smallskip

Recall that, due to assumptions on \(c\) (and taking into account \cite[Example 1.24]{DM}), the functional 
\(\Phi\colon H^1_0(\Omega)\to \ove\R\) given by 
\begin{equation}\label{eq:aux_Phi}
u\mapsto \Phi(u)\coloneqq \int_{\Omega}
c^*\left(x,\frac{|\nabla u|^2}{2}\right)\,\rmd x,
\end{equation}
which appears in the minimization problem \eqref{eq:aux_problem} is lower semicontinuous in the weak \(H^1_0(\Omega)\)-topology. Nevertheless, due to the dependence on the domain variable of the integrand \(c^*\), it might happen that the equality
\begin{equation}\label{eq:Lavrentiev}
\text{the lower semicontinuous envelope of }\Phi|_{\mathscr D(\Omega)}\text{ in the weak } H^1_0(\Omega)\text{-topology}\,=\,\Phi,
\end{equation}
is \emph{not} satisfied and thus also the equality \eqref{eq:aux_problem_relaxed} might not be true. This is related to the occurrence of the so called \emph{Lavrentiev phenomenon}. Since the scope of this paper is beyond analysing the necessary and sufficient conditions on \(c\) in order to avoid such a phenomenon, 
in what follows we will say that 
\begin{center}
Lavrentiev phenomenon \emph{does not} occur for the  functional \(\Phi\) if \eqref{eq:Lavrentiev}
is satisfied. 
\end{center}
For a thorough  discussion about the occurrence of Lavrentiev phenomenon see \cite{MarTreu} and  the references therein. This said, the corresponding theorems to Theorem \ref{thm:auxiliary_problem},  Theorem \ref{thm:aux_related_to_optimal_mu} 
and Theorem \ref{thm:opt_cond_hom} in the heterogeneous case can be stated, respectively, as follows. Their proofs follow verbatim the proofs in the homogeneous case, since the dependence on \(x\) does not have any influence to the arguments used (notice that assumptions on \(c\) with respect to the variable \(x\) are such that all the requirements in the theorems stated in Section \ref{sec:preliminaries} are met; see also Remark \ref{rmk:P1-P3}). 

\begin{thm}[The auxiliary problem in the (SL) heterogeneous case]\label{thm:auxiliary_problem_het}
Let \(f\in H^{-1}(\Omega)\) and assume that 
Lavrentiev phenomenon does not occur for the functional \(\Phi\) defined in \eqref{eq:aux_Phi}.
Then it holds that
\begin{equation}\label{eq:aux_problem_relaxed_het}
\inf_{u\in \mathscr D(\Omega)}\int_{\Omega} c^*\left(x,\frac{|\nabla u(x)|^2}{2}\right)\,\rmd x
-\langle f,u\rangle={\rm I}_{f,c}.
\end{equation}
Moreover, the auxiliary problem in \eqref{eq:aux_problem}
admits a solution and any solution \(\bar u\in H^1_0(\Omega)\) satisfies 
\begin{equation}\label{eq:optimal_pairs_het}
-{\rm div}(\bar v)=f, \quad 
\bar v(x) \in \partial c^*\left(x,\frac{|\nabla \bar u(x)|^2}{2}\right)\nabla \bar u(x)\,\text{ for }\mathcal L^n\text{-a.e.\ }x\in \Omega. 
\end{equation}
\end{thm}

\begin{thm}\label{thm:aux_related_to_optimal_mu_het}
 Let \(f\in H^{-1}(\Omega)\) and assume that Lavrentiev phenomenon does not occur for the functional \(\Phi\) in \eqref{eq:aux_Phi}.
 Then for every optimal conductivity 
\(a_{\rm opt}\in \mathscr M_{\rm opt}\) it holds that 
\[
\inf_{u\mathscr D(\Omega)}\int_{\Omega} \frac{|\nabla u(x)|^2}{2}a_{\rm opt}(x)\,\rmd x-\langle f,u\rangle -\int_{\Omega} c\big(x,a_{\rm opt}(x)\big)\,\rmd x=\min_{u\in H^1_0(\Omega)}\int_{\Omega} c^*\left(x,\frac{|\nabla u(x)|^2}{2}\right)\,\rmd x
-\langle f,u\rangle.
\]
\end{thm}

\begin{thm}[Optimality conditions in the (SL) heterogeneous case]
Let \(f\in H^{-1}(\Omega)\) and assume that 
Lavrentiev phenomenon does not occur for the functional \(\Phi\) in \eqref{eq:aux_Phi}.
Then, the following are equivalent:
the couple \((a_{\rm opt}, \bar u)\in L^1(\Omega)\times H^1_0(\Omega)\)
satisfies the conditions
(referred to as the \emph{optimality conditions})
\begin{equation}\label{eq:OC_SL_het}
\begin{split}
1)\quad &-{\rm div}(a_{\rm opt}
\nabla\bar u)=f \, \text{ in }\mathscr D'(\Omega);\\
2)\quad & a_{\rm opt}(x)\in \partial c^*\left(x,\frac{|\nabla \bar u(x)|^2}{2}\right)\,\text{ for }\mathcal L^n\text{-a.e.\ }x\in \Omega.
\end{split}
\end{equation}
if and only if 
\begin{itemize}
\item [\rm a)] \(a_{\rm opt}\) is a solution of the problem \eqref{eq:mass_opt_pb};
\item [\rm b)] \(\bar u\) is a  solution to the problem  \eqref{eq:aux_problem},
\item [\rm c)] \(\nabla \bar u\) is a solution of the dual problem \eqref{eq:dual_problem_SL} with \(a=a_{\rm opt}\).
\end{itemize}
\end{thm}

\section{The case of linear cost function}\label{sec:linear}
In this section we provide optimality conditions for the optimal conductivities 
\(\mu_{\rm opt}\in \mathscr M_{\rm opt}\) when the cost function \(c\) has a linear growth at infinity.
More specifically,  
we consider throughout this section the following case:
\begin{equation}\tag{L}
\textbf{Linear case:}\qquad {\bf c}^{\infty}\coloneqq\sup_{x\in \ove\Omega}c^\infty(x,1)< +\infty, \quad \text{ where } c^{\infty}(x,1)\coloneqq \big(c(x,\cdot)\big)^{\infty}(1).
\end{equation}

Let us start by observing that the following 
result holds.

\begin{lemma}\label{lem:extension}
Let \(f\in {\rm Adm}_c\) and \(\mu\in {\rm Dom}(\mathcal J_{f,c})\). 
Let us denote by 
\[(-\infty, 0]\ni m\coloneqq \inf_{u\in\mathscr D(\Omega)}\int\frac{|\nabla u|^2}{2}\,\rmd \mu-\langle f, u\rangle\quad \text{ and }\quad C\coloneqq \sqrt{-2m}.\]
Then it holds that 
\begin{equation}\label{eq:stimaextension}
|\langle f,u\rangle|\le 
C\,\|\nabla u\|_{L^2_\mu(\Omega, \R^n)}\quad \text{ for all } u\in \mathscr D(\Omega).
\end{equation}
\end{lemma}

\begin{proof}
Note that for every \(u\in \mathscr D(\Omega)\) we have that 
\[
|\langle f,u\rangle|\leq -m+\|\nabla u\|^2_{L^2_\mu(\Omega, \R^n)}.
\]
Thus, given any \(u\in \mathscr D(\Omega)\), by just plugging in the above inequality \(tu\in \mathscr D(\Omega)\), we have that 
\[
t^2\|\nabla u\|^2_{L^2_\mu(\Omega, \R^n)}-t|\langle f,u\rangle|-m\geq 0 \quad\text { for all }t>0.
\]
Then, optimizing with respect to \(t\),  we obtain precisely \eqref{eq:stimaextension}.
\end{proof}

As a consequence, we have that \(f\) admits a (unique) linear and continuous extension 
\(\hat f\) to the Banach space
\[
S^2_{0,\mu}(\Omega)\coloneqq {\rm cl}_{L^2_{\mu}(\Omega, \R^n)}\big\{\nabla u:\, u\in 
\mathscr D(\Omega)\big\}
\]
endowed with the norm \(\|\cdot\|_{L^2_{\mu}(\Omega, \R^n)}\). The extension \(\hat f\) has the property
\[\hat f[\nabla u]=\langle f, u\rangle\qquad\text{for every }u\in \mathscr D(\Omega).\] 
Thus, a standard relaxation argument gives
\begin{equation}\label{eq:relaxed_pb_of_I}
\inf_{u\in \mathscr D(\Omega)} 
\mathcal F_f(\mu,u)-\mathcal C(\mu)
=\inf_{w\in S^2_{0,\mu}(\Omega)}\int 
\frac{|w|^2}{2}\,\rmd \mu-\hat f[w]-\mathcal C(\mu).
\end{equation}

In order to state and prove the optimality conditions for the optimal conductors that in this are not necessarily  densities with respect to a Lebesgue measure (as in the case (SL)), but may contain also singular parts, in the following subsection we recall the notion of the \(\mu\)-tangential gradient and Sobolev function with respect to a measure \(\mu\in\mathscr M^+(\ove\Omega)\).

\subsection{Sobolev spaces with respect to a measure}\label{sec:Sobolev}
Let us consider a bounded open set \(\Omega\subseteq \R^n\) and a measure \(\mu\in \mathscr 
M^+(\overline\Omega)\). 
In this subsection we recall the notion of Sobolev space \(H^{1}_{0,\mu}(\Omega)\), following the 
approaches in \cite{BBS} and \cite{Zhikov}, which turn out to be equivalent (see \cite{LPR20}). We first fix 
some terminology: given a function \(u\in L^2_\mu(\Omega)\) we say that a vector field \(w\in 
L^2_\mu(\Omega,\R^n)\) is a \(\mu\)-\emph{gradient} of \(u\)
if there exists a sequence \((u_i)_i\subseteq \mathscr D(\Omega)\) such that 
\begin{equation}\label{eq:Sobolev_G}
u_i\to u\, \text{ strongly in }L^2_\mu(\Omega)\quad \text{ and }\quad \nabla u_i\to w\, 
\text{ strongly in }L^2_\mu(\Omega, \R^n).
\end{equation}
We denote the set of all \(\mu\)-gradients of \(u\) by \({\rm G}_\mu(u)\) and define
\[
H^1_{0,\mu}(\Omega)\coloneqq \big\{u\in L^2_\mu(\Omega):\, {\rm G}_\mu(u)\neq \emptyset\big\}.
\]
In particular, notice that every \(w\in {\rm G}_\mu(u)\) for \(u\in H^1_{0,\mu}(\Omega)\)
belongs to the space \(S^2_{0,\mu}(\Omega)\) defined above.
We will use this fact in the proof of Theorem \ref{thm:opt_cond_lin_hom}.

It is not difficult to check that the set \({\rm G}_\mu(u)\) is a closed and convex subset of 
\(L^2_\mu(\Omega, \R^n)\), thus it admits the element of minimal \(L^2\)-norm, denoted by \(\nabla_\mu u\).
The space \(H^1_{0,\mu}(\Omega)\) is a Banach space, when endowed with the norm 
\[
\|u\|_{H^1_{0,\mu}}\coloneqq \|u\|_{L^2_\mu(\Omega)}+\|\nabla_\mu u\|_{L^2_\mu(\Omega, \R^n)}.
\]
Notice that, thanks to Mazur's theorem, we could equivalently require the weak convergence of gradients in 
\eqref{eq:Sobolev_G} instead of the strong one. Also, when \(\mu=\mathcal L^n|_{\Omega}\), due to the 
closability of the norm \(\|u\|_{L^2(\Omega)}+\|\nabla u\|_{L^2(\Omega, \R^n)}\) on \(\mathscr D(\Omega)\),
the above definition reduces to the standard definition of the space of 
Sobolev functions on \(\Omega\) with `zero boundary values'. 
We shall thus in this case use the standard 
notation \(H^1_0(\Omega)\) and the \(\mathcal L^n\)-gradient will be denoted by \(\nabla u\) (this will 
not cause any ambiguity with the same notation of the strong gradient of smooth functions).

Given any positive constant \(M>0\), we set
\begin{equation}\label{eq:LIP}
{\rm LIP}_{0,M}(\Omega)\coloneqq {\rm cl}_{\mathscr C(\overline\Omega)} 
\{u\in \mathscr D(\Omega):\|\nabla u\|_{\infty}\leq M\}.
\end{equation}
We have that ${\rm LIP}_{0,M}(\Omega)$ coincides with the class of Lipschitz functions which vanish on $\partial\Om$ and with their Lipschitz constant bounded by $M$. In particular, \({\rm LIP}_{0,M}(\Om)\subseteq H^1_{0,\mu}(\Omega)\) for every measure $\mu$, and (see for instance \cite{Sobolev_Pekka}, \cite{LPR20}) for every (locally) Lipschitz function $u$ we have
\begin{equation}\label{eq:ug_lip}
|\nabla_\mu u|\leq {\rm lip}(u) \qquad \mu\text{-a.e.\ in }\Omega,
\end{equation}
where
$${\rm lip}(u)(x)\coloneqq\limsup_{\substack{y\to x\\y\neq x}}\frac{|u(x)-u(y)|}{|x-y|}.$$
In particular, for every locally Lipschitz function \(u\in H^1_0(\Omega)\) we have \(|\nabla u|={\rm lip}(u)\, \mathcal L^n\)-a.e.\ in \(\Omega\).

\subsection{Homogeneous cost function}
As in the superlinear case, we shall first concentrate on the homogeneous cost functions, i.e.\ 
we add the following assumption to the cost function \(c\):
\[
\textbf{Homogeneous cost:}\qquad c(x,t)=c(t)\, \text{ for all }t\in \R.
\]

Note that in this case \({\bf c}^\infty=c^\infty(1)\). Let us analyse the auxiliary problem \eqref{eq:aux_pb_2}. 

\begin{thm}[The auxiliary problem in the {\rm (L)} homogenoeus case]\label{thm:aux_problem_linear}
Let \(f\in \mathscr M(\ove\Omega)\). Then 
it holds that 
\begin{equation}\label{eq:aux_problem_relaxed_linear}
\inf_{u\in \mathscr D(\Omega)}\int_{\Omega} c^*\left(\frac{|\nabla u(x)|^2}{2}\right)\,\rmd x
-\langle f,u\rangle={\rm I}_{f,c}.
\end{equation}
Moreover, the problem \eqref{eq:aux_pb_2} admits a solution and any
solution 
\(\bar u\) belongs to the space \(\in {{\rm LIP}_{0,\mathbf c}}(\Omega)\)
with
\(\bfc\coloneqq \sqrt{2c^\infty(1)}\), and satisfies
\begin{equation}\label{eq:optimal_pairs_hom_linear}
-{\rm div}(\bar v)=f, \quad 
\bar v(x) \in \partial c^*\left(\frac{|\nabla \bar u(x)|^2}{2}\right)\nabla \bar u(x)\,\text{ for }\mathcal L^n\text{-a.e.\ }x\in \Omega. 
\end{equation}
\end{thm}
\begin{proof}
The proof  follows along the same lines as in the (SL) case, by just noticing the following:
since \(c(t)\leq c^\infty(1)(t-t_0)\) for all \(t\geq 0\)
and any \(t_0\in {\rm Dom}(c)\) by assumptions,
we
have that \(c^*(t)\geq \bigchi_{\{s:\,s\leq c^\infty(1)\}}(t)+c^\infty(1)t_0\). Thus
every \(u\in H^1_0(\Omega)\) with finite energy \(\int_{\Omega}c^*\big(|\nabla u|^2/2\big)\,\rmd x\) satisfies 
\(|\nabla u(x)|^2\leq 2 c^\infty(1)\) for \(\mathcal L^n\)-a.e.\ \(x\in \Omega\).
Thus, recalling the assumption on \(f\) and the zero boundary condition, any minimizer \(\bar u\) indeed belongs to the space \({\rm LIP}_{0,\bf c}(\Omega)\).
\end{proof}

\begin{thm}\label{thm:aux_related_to_optimal_mu_linear}
 Let \(f\in \mathscr M(\ove\Omega)\). Then \(f\in {\rm Adm}_c\) and for every optimal conductivity 
\(\mu_{\rm opt}\in \mathscr M_{\rm opt}\) it holds that 
\begin{equation}\label{eq:statement}
\inf_{u\in\mathscr D(\Omega)}\int_{\Omega} \frac{|\nabla u(x)|^2}\,\rmd\mu_{\rm opt}-\langle f,u\rangle -\mathcal C(\mu_{\rm opt})=\min_{u\in {\rm LIP_{0,\bf c}}(\Omega)}\int_{\Omega} c^*\left(\frac{|\nabla u(x)|^2}{2}\right)\,\rmd x
-\langle f,u\rangle.
\end{equation}
\end{thm}
\begin{proof}
Notice that (cf. Theorem \ref{thm:min_max_Sorin}) for every \(k\in \N\) it holds that 
\[
\sup_{\mu\in \mathscr K^+_k(\ove \Omega)}\inf_{u\in \mathscr D(\Omega)}\mathcal F_f(\mu, u)-\mathcal C(\mu)
=
\inf_{u\in \mathscr D(\Omega)}\sup_{\mu\in \mathscr K^+_k(\ove \Omega)}\mathcal F_f(\mu, u)-\mathcal C(\mu)
\]
Taking into account the above and Theorem \ref{thm:aux_problem_linear}, and following the same line of proof as in the corresponding theorem in the superlinear case  (Theorem \ref{thm:aux_related_to_optimal_mu}) we show that 
\[
\inf_{u\in H^1_0(\Omega)}\Phi_k(u)-\langle f,u\rangle
\leq
\sup_{\mu\in \mathscr K^+_k(\ove \Omega)}\inf_{u\in \mathscr D(\Omega)}\mathcal F_f(\mu, u)-\mathcal C(\mu)
\leq 
\min_{u\in {\rm LIP}_{0,\bf c}(\Omega)}\int_{\Omega}c^*\left(\frac{|\nabla u|^2}{2}\right)\,\rmd x-\langle f,u\rangle.
\]
In order to send \(k\to \infty\) and conclude with the \(\Gamma\)-convergence argument, just notice that the term 
\(\langle f,\cdot\rangle\) with \(f\in \mathscr M(\ove\Omega)\) can be treated via continuity also in this case, since the functions \(u\in H^1_0(\Omega)\) of finite \(c_k^*\)-energy (with \(k\) large enough) are equi-Lipschitz and, due to zero boundary condition, also equi-bounded, thus Ascoli-Arzel\` a theorem applies.
This proves that 
\[
\sup_{\mu\in \mathscr M(\ove \Omega)}\inf_{u\in \mathscr D(\Omega)}\mathcal F_f(\mu, u)-\mathcal C(\mu)
=
\min_{u\in {\rm LIP}_{0,\bf c}}\int_{\Omega}c^*\left(\frac{|\nabla u|^2}{2}\right)\,\rmd x-\langle f,u\rangle,
\]
showing the admissibility of \(f\) as well as \eqref{eq:statement}.
\end{proof}

\begin{thm}[Optimality conditions in the (L) homogeneous case]\label{thm:opt_cond_lin_hom}
Let \(f\in \mathscr M(\overline\Omega)\).
Then, the following are equivalent:
the couple
\((\mu_{\rm opt}, \bar u)
\in \mathscr M^+(\overline\Omega)
\times {\rm LIP}_{0,\bfc}(\Omega)\) 
satisfies the system of equations 
(referred to as the \emph{optimality conditions})
\begin{equation}\label{eq:OCL}
\begin{split}
1)\quad &-{\rm div}(\mu_{\rm opt}\nabla_{\mu_{\rm opt}}\bar u)
=f\,\text{ in } \mathscr D'(\Omega);\\
2)\quad &\frac{|\nabla_{\mu_{\rm opt}}\bar u(x)|^2}{2}a_{\rm opt}(x)
=c^*\left(\frac{|\nabla\bar u(x)|^2}{2}\right)+
c(a_{\rm opt}(x))\quad\text{holds } 
a_{\rm opt}\mathcal L^n\text{-a.e.\ } x\in \Omega,
;\\
3) \quad &\frac{|\nabla_{\mu_{\rm opt}} \bar u(x)|^2}{2}= 
c^\infty(1)\,\text{ holds for }\mu^s_{\rm opt}\text{-a.e.\ }x\in \Omega;\\
4)\quad & \mu_{\rm opt}(\partial \Omega)=0.
\end{split}
\end{equation}
if and only if 
\begin{itemize}
\item [\rm a)] \(\mu_{\rm opt}\) is a solution of the problem \eqref{eq:MOPc};
\item [\rm b)] \(\bar u\) is a solution to the problem \eqref{eq:aux_pb_2}. 
\end{itemize}
\end{thm}
\begin{proof}
\underline{We first prove \((\Rightarrow)\)}.
Recall that \(\bar u\in {\rm LIP}_{0,\bfc}(\Omega)\subseteq H^1_{0, \mu_{\rm opt}}(\Omega)\). Taking into 
account the hypotheses 1) and 2), we have that
\(\nabla_{\mu_{\rm opt}}\bar u\) is a competitor for
the dual problem to 
\(\inf_{u\in \mathscr D(\Omega)}\mathcal F(\mu_{\rm opt},u)-\mathcal C(\mu_{\rm opt})\) given in \eqref{eq:dual_problem_SL}.
Take now any sequence \((u_i)_{i\in \N}\subseteq \mathscr D(\Omega)\) such that 
\(u_i\to \bar u\) in \(H^1_0(\Omega)\) as \(i\to +\infty\),
\[
\lim_{i\to +\infty}\int_{\Omega}c^*\left(\frac{|\nabla u_i|^2}{2}\right)\,\rmd x=\int_{\Omega}c^*\left(\frac{|\nabla \bar u|^2}{2}\right)\,\rmd x\quad \text{ and }\quad \sup_{i\in \N}|\nabla u_i|\leq \bfc
\] 
(whose existence is granted by Theorem \ref{thm:aux_problem_linear} and Theorem \ref{thm:aux_related_to_optimal_mu_linear}).
Then, up to a subsequence we have that \(u_i\to \bar u\) 
weakly in \(L^2_{\mu_{\rm opt}}(\Omega)\) and \(\nabla u_i\to \bar w\) weakly in 
\(L^2_{\mu_{\rm opt}}(\Omega, \R^n)\), for some \(\bar w\in L^2_{\mu_{\rm opt}}(\Omega, \R^n)\).
In particular, \(\bar w\in {\rm G}_{\mu_{\rm opt}}(\bar u)\).
Then due to the hypotheses 1) and 2), we get (as in the proof of Theorem \ref{thm:opt_cond_hom} above)
\(\langle f, \bar u\rangle=\int \nabla_{\mu_{\rm opt}}\bar u\cdot \bar w\,\rmd \mu_{\rm opt}.\)
Due to the lower semiconitnuity of the norm we further get 
\[
\int \frac{|\bar w|^2}{2}\,\rmd \mu_{\rm opt}
\leq \lim_{i\to +\infty}\int \frac{|\nabla u_i|^2}{2}\,\rmd \mu_{\rm opt}
\leq\lim_{i\to +\infty}\int_{\Omega}c^*\left(\frac{|\nabla u_i|^2}{2}\right)\,\rmd x
+\mathcal C(\mu_{\rm opt}) 
= \int \frac{|\nabla_{\mu_{\rm opt}}\bar u|^2}{2}\,\rmd \mu_{\rm opt},
\]
where the last equality follows from the hypotheses 3) and 4).
In particular, due to the minimality of 
\(\nabla_{\mu_{\rm opt}}\bar u\),
we deduce that \(\bar w=\nabla_{\mu_{\rm opt}}\bar u\) holds \(\mu_{\rm opt}\)-a.e.
Therefore, it follows that 
\[
\begin{split}
-\int \frac{|\nabla_{\mu_{\rm opt}} \bar u|^2}{2}\,\rmd \mu_{\rm opt}-\mathcal C(\mu_{\rm opt})
&\overset{\eqref{eq:dual_problem_SL}}{\leq}
\inf_{u\in \mathscr D(\Omega)}
\mathcal F(\mu_{\rm opt},u)-\mathcal C(\mu_{\rm opt})\\
&\leq \,{\rm I}_{f,c}\\
&\leq 
 \int c^*\left(\frac{|\nabla \bar u(x)|^2}{2}\right)\,\rmd x-\langle f,\bar u\rangle\\
&= \int \frac{|\nabla_{\mu_{\rm opt}} \bar u|^2}{2}\,\rmd  \mu_{\rm opt} - \langle f,\bar u\rangle 
-\mathcal C(\mu_{\rm opt})\\
&=-\int \frac{|\nabla_{\mu_{\rm opt}} \bar u|^2}{2}\,\rmd \mu_{\rm opt}-\mathcal C(\mu_{\rm opt}),
\end{split}
\]
which shows that all the inequalities above are indeed equalities, and thus proves a) and b).\\

\underline{Let us prove now \((\Leftarrow)\)}.
Pick a sequence 
\((u_i)_{i\in \N}\subseteq \mathscr D(\Omega)\cap {\rm LIP}_{0,\bfc}(\Omega)\) 
converging weakly in \(H^1_0(\Omega)\) to \(\bar u\) and such that
\[
\int_{\Omega}c^*\left(\frac{|\nabla u_i|^2}{2}\right)\,\rmd x\to \int_{\Omega} c^*\left(\frac{|\nabla \bar u|^2}{2}\right)\,\rmd x,\quad
\text{ as }i\to \infty.
\]
Thus, up to a (non-relabeled) subsequence we have that there 
is \(\bar w\in L^2_{\mu_{\rm opt}}(\Omega)\)
\[
u_i\rightharpoonup \bar u\, \text{ weakly in } L^2_{\mu_{\rm opt}}(\Omega)
\quad \text{ and }\quad \nabla u_i\rightharpoonup \bar w\,\text{ weakly in } L^2_{\mu_{\rm opt}}(\Omega).
\]
In particular, \(\bar w\in {\rm G}_{\mu_{\rm opt}}(\bar u)\) and, due to the lower semicontinuity of 
the norm we have that
\[
\int \frac{|\bar w|^2}{2}\,\rmd  \mu_{\rm opt}\leq \int_{\Omega}c^*\left(\frac{|\nabla \bar u|^2}{2}\right)\,
\rmd x +\mathcal C(\mu_{\rm opt}).
\]
Due to the hypotheses a), the quantity in the  \eqref{eq:dual_problem_SL} is finite; 
denote by \(\bar \sigma\) a solution for the dual problem in \eqref{eq:dual_problem_SL}.
Then, taking the above into account,  we have  that 
\(\langle f,\bar u\rangle
=\int \bar \sigma\cdot \bar w\,\rmd \mu_{\rm opt}\).
Finally, by using the equality 
\[
-\int \frac{|\bar \sigma|^2}{2}\,\rmd \mu_{\rm opt} -\mathcal C(\mu_{\rm opt})
\overset{\eqref{eq:dual_problem_SL}}{=}
\inf_{u\in \mathscr D(\Omega)}\mathcal F(\mu_{\rm opt}, u)
={\rm I}_{f,c}
=\int_{\Omega}c^*\left(\frac{|\nabla \bar u|^2}{2}\right)\,\rmd x-\langle f,\bar u\rangle,
\]
together with all the above considerations
we deduce that  \(\bar w=\bar \sigma\) holds 
\(\mu_{\rm opt}\)-a.e.\ and that 
\[
\int \frac{|\bar w|^2}{2}\,\rmd \mu_{\rm opt} -
\mathcal C(\mu_{\rm opt})
=\int_{\Omega}c^*\left(\frac{|\nabla \bar u|^2}{2}\right)\,\rmd x.
\]
Given that \(\bar u\in {\rm LIP}_{0,\bfc}(\Omega)\subseteq H^1_{0,\mu_{\rm opt}}(\Omega)\), 
its \(\mu\)-gradient \(\nabla_{\mu_{\rm opt}}\bar u\) is a competitor for the relaxed problem of 
\(\inf_{u\in \mathscr D(\Omega)}\mathcal F(\mu_{\rm opt}, u)\) given in \eqref{eq:relaxed_pb_of_I}.
Also, being \(\nabla_{\mu_{\rm opt}}\bar u\) 
element of \({\rm G}_{\mu_{\rm opt}}(\bar u)\) 
of the minimal \(L^2_{\mu_{\rm opt}}\)-norm, we have that 
\[
\begin{split}
\int \frac{|\bar \sigma|^2}{2}\,\rmd \mu_{\rm opt} -\langle f,\bar u\rangle -\mathcal C(\mu_{\rm opt})
&\geq \int \frac{|\nabla_{\mu_{\rm opt}}\bar u|^2}{2}\,\rmd \mu_{\rm opt} 
-\langle f,\bar u\rangle-\mathcal C(\mu_{\rm opt})\\
&\overset{\eqref{eq:relaxed_pb_of_I}}{\geq} \inf_{u\in \mathscr D(\Omega)}\mathcal F(\mu_{\rm opt}, u)={\rm I}_{f,c}
=\int_{\Omega}c^*\left(\frac{|\nabla \bar u|^2}{2}\right)\,\rmd x-\langle f,\bar u\rangle,
\end{split}
\]
proving that \(\bar\sigma=\nabla_{\mu_{\rm opt}}\bar u\) holds \(\mu_{\rm opt}\)-a.e.. Thus the item 
1) follows, as well as the 
validity of the condition 2), 3) and 4) in an integral form. Namely, it holds that 
\begin{equation}
\label{eq:integral_conditions_linear}
\begin{split}
&\int_{\Omega}\frac{|\nabla_{\mu_{\rm opt}}\bar u(x)|^2}{2}a_{\rm opt}(x)\,\rmd x=\int_{\Omega}c^*\left(\frac{|\nabla\bar u(x)|^2}{2}\right)\,\rmd x+\int_{\Omega}c(a_{\rm opt}(x))\,\rmd x,\\
&\int_{\Omega}\frac{|\nabla_{\mu_{\rm opt}}\bar u|^2}{2}\,\rmd \mu_{\rm opt}^s=c^\infty(1)\mu_{\rm opt}^s(\ove\Omega).  
\end{split}
\end{equation}

To get the conditions 2), 3) and 4) pointwise, 
we argue as follows: 
In order to
get 2), 
recall that \(|\nabla \bar u|={\rm lip}(\bar u)\)  (see Subsection \ref{sec:Sobolev})
holds \(\mathcal L^n\)-a.e.\ in \(\Omega\) 
(and thus \(a_{\rm opt}\mathcal L^n\)-a.e.\ in \(\Omega\)).
On the other hand, we have that \(|\nabla_{\mu_{\rm opt}}\bar u|\leq {\rm lip}(\bar u)\) holds \(\mu_{\rm opt}\)-a.e.\ 
(and thus \(a_{\rm opt}\mathcal L^n\)-a.e.\ in \(\Omega\)). Hence we deduce 
that \(|\nabla_{\mu_{\rm opt}}\bar u|\leq|\nabla \bar u|\) holds \(a_{\rm opt}\mathcal L^n\)-a.e.\ in \(\Omega\).
This property, together with the conjugate inequality gives
\[
c^*\left(\frac{|\nabla \bar u(x)|^2}{2}\right) +
c\big(a_{\rm opt}(x)\big)\geq \frac{|\nabla \bar u(x)|^2}{2}a_{\rm opt}(x)
\geq \frac{|\nabla_{\mu_{\rm opt}}\bar u(x)|^2}{2}a_{\rm opt}(x),\quad a_{\rm opt}\mathcal L^n\text{-a.e.\ }x\in \Omega,
\]
and together with the first line in \eqref{eq:integral_conditions_linear} 
provide us with the claimed pointwise version 2). 
\smallskip

The conditions 3) and 4) are due to the fact that \(\nabla_{\mu_{\rm opt}}\bar u\) is  a weak \(L^2_{\mu_{\rm opt}}\) limit  of a sequence of 
smooth gradients \(\nabla u_i\) satisfying the property \(\frac{|\nabla u_i(x)|^2}{2}\leq c^\infty(1)\) 
for all \(x\in \Omega\). Thus, up to applying Mazzur's lemma and getting the strong (and thus, up to a 
subsequence, pointwise) convergence, we have that
\(\frac{|\nabla_{\mu_{\rm opt}} \bar u(x)|^2}{2}\leq c^\infty(1)\) for \(\mu_{\rm opt}\)-a.e.\ \(x\in \Omega\). 
This together with \eqref{eq:integral_conditions_linear} 
implies that 
\(\mu_{\rm opt}(\partial \Omega)
=\mu^s_{\rm opt}(\partial \Omega)=0\) 
and that 
\(\frac{|\nabla_{\mu_{\rm opt}} \bar 
u(x)|^2}{2}=c^\infty(1)\) holds for \(\mu^s_{\rm opt}\)-a.e.\ \(x\in \Omega\), as claimed.
This concludes the proof.
\end{proof}
\subsection{Heterogeneous cost function}
In this section we consider a general \(x\)-dependent cost function \(c\), namley
\[
\textbf{Heterogeneous cost:}\qquad c\colon \ove\Omega\times \R\to [0,+\infty]\, \text{ satisfying } {\rm (P1)}-{\rm (P3)}.
\]
As already explained in the Subsection \ref{sec:Heterogeneous_SL}, when the cost function is \(x\)-dependent, Lavrentiev phenomenon may occur. Under the assumption 
of no occurrence of such phenomenon, we have 
the following corresponding results to the Theorem \ref{thm:aux_problem_linear}, Theorem \ref{thm:aux_related_to_optimal_mu_linear} and Theorem \ref{thm:opt_cond_lin_hom}, respectively. Proofs follow exactly the same line as in the homogeneous case.
\smallskip

As in the (SL) case, we first consider the auxiliary problem \eqref{eq:aux_problem}.

\begin{thm}[The auxiliary problem in the (L) heterogeneous case]
Let \(f\in \mathscr M(\ove\Omega)\) and assume that Lavrentiev phenomenon does not occur for the functional \(\Phi\) defined in \eqref{eq:aux_Phi}. Then it holds that
\begin{equation}\label{eq:aux_problem_relaxed_linear_het}
\inf_{u\in \mathscr D(\Omega)}\int_{\Omega} c^*\left(x,\frac{|\nabla u(x)|^2}{2}\right)\,\rmd x
-\langle f,u\rangle={\rm I}_{f,c}.
\end{equation}
Moreover, the auxiliary problem \eqref{eq:aux_problem} admits a solution and any solution \(\bar u\) belongs to the space
\({\rm LIP}_{0,\mathbf c}(\Omega)\)
with 
\(\bfc\coloneqq \sqrt{2{\bf c}^\infty}\), and satisfies
\begin{equation}\label{eq:bound_for_nabla_u}
|\nabla \bar u(x)|\leq \sqrt{2c^\infty(x,1)}
\quad
\text{ for } \mathcal L^n\text{-a.e.\ } x\in \Omega
\end{equation} and
\begin{equation}\label{eq:optimal_pairs_het_linear}
-{\rm div}(\bar v)=f, \quad 
\bar v(x) \in \partial c^*\left(x,\frac{|\nabla \bar u(x)|^2}{2}\right)\nabla \bar u(x)\,\text{ for }\mathcal L^n\text{-a.e.\ }x\in \Omega. 
\end{equation}
\end{thm}
\begin{thm}\label{thm:aux_related_to_optimal_mu_het_lin}
Let \(f\in \mathscr M(\ove\Omega)\) and assume that Lavrentiev phenomenon does not occur for the functional \(\Phi\) defined in \eqref{eq:aux_Phi}. Then for every optimal conductivity 
\(\mu_{\rm opt}\in \mathscr M_{\rm opt}\) it holds that 
\[
\inf_{u\in \mathscr D(\Omega)}\int_{\Omega} \frac{|\nabla u(x)|^2}{2}\,\rmd\mu_{\rm opt}-\langle f,u\rangle -\mathcal C(\mu_{\rm opt})=\min_{u\in {\rm LIP}_{0,\bf c}(\Omega)}\int_{\Omega} c^*\left(x,\frac{|\nabla u(x)|^2}{2}\right)\,\rmd x
-\langle f,u\rangle.
\]
\end{thm}
\begin{thm}
Let \(f\in \mathscr M(\overline\Omega)\) and assume that Lavrentiev phenomenon does not occur for the functional \(\Phi\) defined in \eqref{eq:aux_Phi}.
Then, the following are equivalent:
the couple
\((\mu_{\rm opt}, \bar u)
\in \mathscr M^+(\overline\Omega)
\times {\rm LIP}_{0,\bfc}(\Omega)\) 
satisfies the system of equations 
(referred to as the \emph{optimality conditions})
\begin{equation}\label{eq:OCL_het}
\begin{split}
1)\quad &-{\rm div}(\mu_{\rm opt}\nabla_{\mu_{\rm opt}}\bar u)
=f\,\text{ in } \mathscr D'(\Omega);\\
2)\quad &\frac{|\nabla_{\mu_{\rm opt}}\bar u(x)|^2}{2}a_{\rm opt}(x)
=c^*\left(x,\frac{|\nabla\bar u(x)|^2}{2}\right)+
c(x, a_{\rm opt}(x))\quad\text{holds } 
a_{\rm opt}\mathcal L^n\text{-a.e.\ } x\in \Omega,
;\\
3) \quad &\frac{|\nabla_{\mu_{\rm opt}} \bar u(x)|^2}{2}= 
c^\infty(x, 1)\,\text{ holds for }\mu^s_{\rm opt}\text{-a.e.\ }x\in \Omega;\\
4)\quad & \mu_{\rm opt}(\partial \Omega)=0.
\end{split}
\end{equation}
if and only if 
\begin{itemize}
\item [\rm a)] \(\mu_{\rm opt}\) is a solution of the problem \eqref{eq:mass_opt_pb};
\item [\rm b)] \(\bar u\) is a solution to the problem  \eqref{eq:aux_problem}. 
\end{itemize}
\end{thm}
\section{Examples and variants of the problem}\label{sec:Examples}

In this section, we provide some concrete examples of interest. 

\begin{example}\label{ex:c_quadratic}
Consider the function \(c(t)\coloneqq t^2/2\) for \(t\ge0\), and set \(c(t)=+\infty\) elsewhere in \(\R\). This is a superlinear case {\rm (SL)}, with \(c^\infty(1)\equiv+\infty\) and
$$c^*(s)=\begin{cases}
s^2/2&\text{if }s\ge0\\
0&\text{if }s<0.
\end{cases}$$
In this case the auxiliary variational problem becomes
$$\min\left\{\int_\Om\Big(\frac{|\nabla u|^4}{8}-f\,u\Big)\,\rmd x\ :\ u\in W^{1,4}_0(\Om)\right\}$$
and its unique solution $\bar u$ is determined by the nonlinear PDE
$$\begin{cases}
-\Delta_4u=2f\\
u\in W^{1,4}_0(\Om).
\end{cases}$$
Then the optimal conductivity coefficient $a_{\rm opt}$ is in $L^2(\Om)$ and is given by
$$a_{\rm opt}=|\nabla\bar u|^2/2\,,$$
and the coupling between $\bar u$ and $a_{\rm opt}$ is through the PDE
$$-{\rm div}(a_{\rm opt}\nabla\bar u)=f.$$
Note that in this case the right-hand side $f$ can be taken in $W^{-1,4/3}(\Om)$, which allows Dirac masses as soon as $n\le3$.
For instance, if $\Om$ is the unit ball of $\R^d$ and $f\equiv1$, we obtain
$$\bar u(x)=\frac34\Big(\frac2n\Big)^{1/3}\big(1-|x|^{4/3}\big),$$
which gives
$$a_{\rm opt}(x)=\frac12\Big(\frac2n\Big)^{2/3}|x|^{2/3}.$$
Taking, with the same $\Om$, $f=\delta_0$ (with $n\le3$) gives
$$\bar u(x)=\frac{3\cdot2^{1/3}}{4-n}\big(1-|x|^{(4-n)/3}\big),$$
and
$$a_{\rm opt}(x)=2^{-1/3}|x|^{2(1-n)/3}.$$
\end{example}

\begin{example}
Consider the cost function \(c(t)\coloneqq t+\frac{1}{t}\) for \(t>0\), and set \(c(t)=+\infty\) elsewhere in \(\R\). This corresponds to a case when both the costs of materials with large and small conductivities are high. This situation falls in the case {\rm (L)} of linear growth, and we have \(c^{\infty}(1)=1\). An easy calculation gives the conjugate function
$$c^*(s)=\begin{cases}
-2\sqrt{1-s}&\text{if }s\le1\\
+\infty&\text{otherwise.}
\end{cases}$$
Therefore the auxiliary variational problems becomes
$$\min\left\{\int_\Om-2\sqrt{1-\frac{|\nabla u|^2}{2}}\,\rmd x-\langle f,u\rangle:\ u\in {\rm LIP}_{0,\sqrt2}(\Om)\right\},$$
and it has a unique solution $\bar u$. Then we can recover the optimal measure $\mu_{\rm opt}=a_{\rm opt}
\mathcal L^n+\mu^s_{\rm opt}$ by Theorem \ref{casoL}:
\begin{equation}\label{recov}
a_{\rm opt}=\Big(1-\frac{|\nabla\bar u|^2}{2}\Big)^{-1/2}\qquad\mathcal L^n\text{-a.e. on }\{x\in\Omega:\, a_{\rm opt}(x)>0\}.
\end{equation}
Concerning the singular part $\mu^s_{\rm opt}$, we have
$$|\nabla_{\mu_{\rm opt}}\bar u|=\sqrt2\qquad\mu^s\text{-a.e.}$$
and the coupling between $\bar u$ and $\mu_{\rm opt}$ is through the PDE
$$-{\rm div}_{\mu_{\rm opt}}(\nabla_{\mu_{\rm opt}}\bar u)=f.$$
Note that, since $c(0)=+\infty$, we have $a_{\rm opt}(x)>0$ for a.e. $x\in\Om$, and by \eqref{recov} this implies $a_{\rm opt}(x)\ge1$ a.e. on $\Om$.

\end{example}

\begin{example}[Comparison with \cite{BB2001}]
Let us consider the homogeneous cost function \(c(t)=\frac{1}{2}t\) for all \(t\geq 0\) and set to be 
\(+\infty\) elsewhere in \(\R\).
Clearly, we are in the case {\rm (L)}, with \(c^{\infty}(1)=\frac{1}{2}\) and 
\(c^*=\bigchi_{[-\infty,1/2]}\).
Let \(\mu_{\rm opt}\) and \(\bar u\) be solutions of \eqref{eq:MOPc} and \eqref{eq:aux_pb_2}, respectively.
Note that the point 2) in \eqref{eq:OCLintro} gives us that 
\[
\frac{|\nabla_{\mu_{\rm opt}}\bar u|^2}{2}=\frac{|\nabla \bar u|^2}{2}=\frac{1}{2},\quad a_{\rm opt}\mathcal 
L^n\text{-a.e.\ in } \Omega,
\]
while by the point 3) we have that \(\frac{|\nabla_{\mu_{\rm opt}}\bar u|^2}{2}=\frac{1}{2}\) holds 
\(\mu_{\rm opt}^s\)-a.e.\ in \(\Omega\).
All in all, we get that \(|\nabla_{\mu_{\rm opt}} \bar u|=1\) holds \(\mu_{\rm opt}\)-a.e.\ in \(\Omega\). 
This, together with the conditions 1), 4) and 5) gives 
precisely the optimality conditions given in \cite[{Equation} (4.1)]{BB2001} with \(\Sigma=\partial \Omega\).
On the other hand, it is clear that the 
couples \((\mu_{\rm opt}, \bar u)\) satisfying  \cite[{Equation} (4.1)]{BB2001} in \(\Omega\) and with 
\(\Sigma=\partial\Omega\)
satisfy also the optimality conditions given in \eqref{eq:OCLintro}.
\end{example}
\begin{remark}
All the stated results can be obtained, by means of the same techniques, also in the case of energies \(\mathcal F(\mu,u)\) involving 
\(|\nabla u|^p\) for any \(1<p<+\infty\). It would be further interesting to investigate, in light of recent results in \cite{GL}, the case \(p=1\).
\end{remark}

\bigskip

\noindent{\bf Acknowledgments.} The work of GB and DL is part of the PRIN 2017 {\it``Gradient flows, Optimal 
Transport and Metric Measure Structures''}, and the work of MSG and DL is part of the PRIN 2017 
{\it``Variational Methods for Stationary and Evolution Problems with Singularities and Interfaces''}, both 
funded by the Italian Ministry of Research and University. The authors are member of the Gruppo Nazionale 
per l'Analisi Matematica, la Probabilit\`a e le loro Applicazioni (GNAMPA) of the Istituto Nazionale di Alta 
Matematica (INdAM).

\bigskip

\bibliographystyle{siam}

\bigskip
{\small\noindent
Giuseppe Buttazzo:
Dipartimento di Matematica,
Universit\`a di Pisa\\
Largo B. Pontecorvo 5,
56127 Pisa - ITALY\\
{\tt giuseppe.buttazzo@dm.unipi.it}\\
{\tt http://www.dm.unipi.it/pages/buttazzo/}

\bigskip\noindent
Maria Stella Gelli:
Dipartimento di Matematica,
Universit\`a di Pisa\\
Largo B. Pontecorvo 5,
56127 Pisa - ITALY\\
{\tt maria.stella.gelli@unipi.it}

\bigskip\noindent
Danka Lu\v ci\' c:
Dipartimento di Matematica,
Universit\`a di Pisa\\
Largo B. Pontecorvo 5,
56127 Pisa - ITALY\\
{\tt danka.lucic@dm.unipi.it}
}

\end{document}